\documentclass[a4paper,11pt]{article}
\usepackage{amsmath,amsthm,amssymb}
\usepackage{amsfonts}
\usepackage{amssymb}
\usepackage{newlfont}
\usepackage{color}
\usepackage[bookmarks=false]{hyperref}

\numberwithin{equation}{section}
{\theoremstyle{definition}\newtheorem{definition}{Definition}[section]

\newtheorem{remark}[definition]{Remark}

}
\newtheorem{proposition}[definition]{Proposition}
\newtheorem{lemma}[definition]{Lemma}
\newtheorem{theorem}[definition]{Theorem}

\newcommand{\M}{\operatorname{M}}

\newcommand{\rl}{\mathord{\text{\rm L}}}
\newcommand{\aut}{\operatorname{Aut}}

\newcommand{\id}{\mathord{\operatorname{id}}}
\newcommand{\si}{\sigma}
\newcommand{\Tr}{\operatorname{Tr}}

\newcommand{\ad}{\operatorname{Ad}}

\newcommand{\Mtil}{\widetilde{M}}

\newcommand{\stab}{\operatorname{Stab}}
\newcommand{\psl}{\operatorname{PSL}}

\newcommand{\Sp}{\operatorname{Sp}}

\newcommand{\SO}{\operatorname{SO}}
\newcommand{\SU}{\operatorname{SU}}

\newcommand{\op}{^\text{\rm op}}
\newcommand{\cb}{_\text{\rm cb}}

\newcommand{\mtil}{\widetilde{M}}

\newcommand{\bim}[3]{\mathord{\raisebox{-0.4ex}[0ex][0ex]{\scriptsize $#1$}{#2}\hspace{-0.05ex}\raisebox{-0.4ex}[0ex][0ex]{\scriptsize $#3$}}}

\newcommand{\cz}{\mathbb{C}} \newcommand{\g}{\mathcal{G}}
\newcommand{\nz}{\mathbb{N}} \newcommand{\uu}{\mathcal{U}}
\newcommand{\zz}{\mathbb{Z}} 
 \newcommand{\rz}{\mathbb{R}}
\newcommand{\ff}{\mathcal{F}} \newcommand{\hh}{\mathcal{H}}
 
\newcommand{\kk}{\mathcal{K}} \def \mm {\mathcal{M}}
\newcommand{\arn}{\mathcal{A}} \newcommand{\nn}{\mathcal{N}}
 \def \al {\alpha}

\def\La{\Lambda}\def\om{\omega}
\def\de{\delta} \def\be{\beta}\def\De{\Delta}
\newcommand{\norm}[1]{\left\Vert#1\right\Vert}
\newcommand{\abs}[1]{\left\vert#1\right\vert}
\def \it {\textit} \def \te {\theta}\def \bf {\textbf}
\def \sset {\subset}
\def \rm {\textrm}\def \tens{\mathbin{\overline{\otimes}}} \def \<{\langle} \def \>{\rangle}
\newcommand{\map}{\longmapsto}
\newcommand{\ds}{\displaystyle}
\def \fy {\varphi}
\def \pr {\prime} \def \bi {\prime\prime}\def \act {\curvearrowright}
\newcommand{\lf}{\left}
\newcommand{\rg}{\right}

\newcommand{\eps}{\varepsilon}

\newcommand{\T}{\mathbb{T}}
\def \aut {\textrm{Aut}}  \def \id {\textrm{id}}
\def \barr {\begin{array}} \def \earr {\end{array}}

\begin{document}

\title{\boldmath\bf {W$^*$-superrigidity for wreath products with groups having positive first $\ell^2$-Betti number }}
\author{Mihaita Berbec}
\date{}
\maketitle

\begin{abstract}\noindent
In \cite{BV12} we have proven that, for all hyperbolic groups and for all non-trivial free products $\Gamma$, the left-right wreath product group $\mathcal{G}:=(\zz/2\zz)^{(\Gamma)}\rtimes (\Gamma\times\Gamma)$ is W$^\ast$-superrigid, in the sense that its group von Neumann algebra $\rl\mathcal{G}$ completely remembers the group $\mathcal{G}$. In this paper, we extend this result to other classes of countable groups. More precisely, we prove that for weakly amenable groups $\Gamma$ having positive first $\ell^2$-Betti number, the same wreath product group $\g$ is W$^\ast$-superrigid.
\end{abstract}

\section{Introduction and main result}

To any countable discrete group $\Gamma$ we can associate the group von Neumann algebra $\rl\Gamma$ generated by the image of the left regular representation of $\Gamma$ on the Hilbert space $\ell^2(\Gamma)$. This construction goes back to Murray and von Neumann \cite{MvN43} and provides a very rich source of examples of von Neumann algebras. The most interesting case is when $\rl\Gamma$ has trivial center, corresponding to $\Gamma$ having infinite conjugacy classes (i.c.c.), i.e. $\Gamma$ is infinite and all of its conjugacy classes, except for the trivial one, are infinite. In this case, $\rl\Gamma$ is a II$_1$ factor, i.e. an infinite dimensional von Neumann algebra that has trivial center and admits a positive trace.

One of the main problems in the theory of von Neumann algebras is to classify the group factors $\rl\Gamma$ in terms of the group $\Gamma$. More precisely, we are interested in answering the following question: does the group factor $\rl\Gamma$ \emph{remember} the group $\Gamma$? This natural question leads to two important concepts: \emph{softness}, when $\rl\Gamma$ does not remember the group $\Gamma$, and \emph{rigidity}, when
$\rl\Gamma$ completely remembers the group $\Gamma$. There is a long list of examples of groups that are soft. The celebrated theorem of Connes \cite{Co76} says that all group II$_1$ factors arising from i.c.c. amenable groups are isomorphic to the hyperfinite II$_1$ factor. This shows that amenable groups manifest a remarkable softness: all the algebraic properties of the group, except their amenability, are lost when we pass to the group von Neumann algebra. In \cite{Dy93}, Dykema proved that for $\Gamma_1,\ldots,\Gamma_n$ infinite amenable groups, $n\ge 2$, the
group von Neumann algebra of their free product $\rl(\Gamma_1\ast\ldots\ast\Gamma_n)$ is isomorphic to the free group factor $\rl\mathbb{F}_n$. Ioana and Bowen obtained the first results saying that plain wreath products tend to be soft, namely all $\rl(\mathbb{F}_n\wr\zz)$, for $n\ge 2$, are isomorphic \cite{Io06} and all $\rl(H\wr\mathbb{F}_2)$, for $H$ non-trivial abelian, are isomorphic \cite{Bo09a},\cite{Bo09b}. Moreover, in \cite{IPV10}, Ioana, Popa and Vaes proved that there exist infinitely many non-isomorphic countable groups $\Lambda$ such that $\rl\Lambda$ is
isomorphic to $\rl(\zz/2\zz\wr \psl(n,\zz))$, for $n\ge 2$.

On the other hand, it is a famous open problem whether the free group factors $\rl\mathbb{F}_n$, with $n \ge 2$, are isomorphic or not. Another big open problem is \emph{Connes' rigidity conjecture}. In \cite{Co80a},\cite{Co80b}, Connes asked whether two i.c.c. property (T) groups $\Gamma$ and
$\Lambda$, with isomorphic group von Neumann algebras $\rl\Gamma \cong \rl \Lambda$, must necessarily be isomorphic. This conjecture remains wide open, even for classical groups like $\psl(n,\zz)$, with $n\ge 3$. Remark, however, that by \cite{CH89}, whenever $\Gamma$ and $\Lambda$ are lattices in $\Sp(n, 1)$, respectively $\Sp(m, 1)$, the isomorphism $\rl\Gamma\cong\rl\Lambda$ implies that $n = m$.

In \cite{IPV10}, Ioana, Popa and Vaes established the first \emph{W$^\ast$-superrigidity} theorem for group von Neumann algebras: for a large class of generalized wreath product groups $G = (\zz / 2\zz)^{(I)} \rtimes \Gamma$, it was shown that if $\rl G \cong \rl \Lambda$, for an arbitrary countable group $\Lambda$, then $G$ must be isomorphic with $\Lambda$. Such a group $G$ is said to be \emph{W$^\ast$-superrigid} (see Definition \ref{def.Wstar-superrigid}), and in this case the group von Neumann algebra $\rl G$ completely remembers $G$.

Following the same strategy as in \cite{IPV10}, we have proven in \cite{BV12} that the more natural \emph{left-right wreath product} groups $\g = (\zz / 2\zz)^{(\Gamma)} \rtimes (\Gamma \times \Gamma)$ are W$^\ast$-superrigid, where the direct product $\Gamma \times \Gamma$ acts on $\Gamma$ by left-right multiplication, and where $\Gamma$ is either the free group $\mathbb{F}_n$, with $n \geq 2$, or any i.c.c. hyperbolic group, or any non-trivial free product $\Gamma_1 \ast \Gamma_2$.

In this paper, we enlarge the class of groups covered by \cite{BV12}, proving that the left-right wreath product $\g = (\zz / 2\zz)^{(\Gamma)} \rtimes (\Gamma \times \Gamma)$ is W$^\ast$-superrigid whenever $\Gamma$ belongs to a certain class of groups with positive first $\ell^2$-Betti number or it is a certain non-trivial amalgamated free product or HNN extension (see Theorem \ref{thm.main-intro}).

In order to state our main theorem, we first need to introduce a few notions. Recall from \cite{CH89} that a countable group $\Gamma$ is said to be \emph{weakly amenable} if it admits a sequence of finitely supported functions $\fy_n : \Gamma \to \cz$ tending to $1$ pointwise and satisfying $\sup_n \norm{\fy_n}\cb < \infty$, where $\norm{\fy}\cb$ denotes the Herz-Schur norm of $\fy$ (i.e. the cb-norm of the linear map $\rl\Gamma \ni u_g \mapsto \fy(g) u_g \in \rl\Gamma$).

If $\Gamma$ is a countable group and $\pi:\Gamma\to \mathcal{O}(\kk_\rz)$ is an orthogonal representation of $\Gamma$ on a real Hilbert space $\kk_\rz$, then a \emph{1-cocycle} $c$ into $\pi$ is a map $c:\Gamma\to \kk_\rz$ satisfying the following 1-cocycle relation: $$c(gh)=c(g)+\pi(g)c(h), \rm{ for all }g,h \in \Gamma.$$

A subgroup $\Sigma<\Gamma$ is called \emph{malnormal} if $\Sigma\cap g\Sigma g^{-1}=\{1\}$, for all $g\in \Gamma\setminus\Sigma$.  A subgroup $\Sigma<\Gamma$ is said to be \emph{relatively malnormal} if there exists an infinite index subgroup $\Lambda<\Gamma$ such that $\Sigma\cap g\Sigma g^{-1}$ is finite, for all $g\in\Gamma\setminus\Lambda$. If $\{\Sigma_i\}_{i\in I}$ is a family of subgroups of $\Gamma$, then we say that $\{\Sigma_i\}_{i\in I}$ is \emph{malnormal} in $\Gamma$ if $g\Sigma_i g^{-1}\cap \Sigma_j=\{1\}$, unless $i=j$ and $g\in \Sigma_i$.

If $\Gamma$ is a countable group, $\Sigma<\Gamma$ is a subgroup and $\theta:\Sigma\to\Gamma$ is an injective group homomorphism, then the HNN extension HNN$(\Gamma,\Sigma,\theta)$ is the group generated by a copy of $\Gamma$ and an extra generator $t$, called stable letter, subject to relations $t g t^{-1}=\theta(g)$, for all $g\in \Sigma$. We say that HNN$(\Gamma,\Sigma,\theta)$ is \emph{non-degenerate} if $\Sigma\neq\Gamma\neq\theta(\Sigma)$. Note that, in this case, HNN$(\Gamma,\Sigma,\theta)$ contains a copy of the free group on two generators, hence it is non-amenable. In the same spirit, we say that an amalgamated free product $\Gamma=\Gamma_1\ast_\Sigma\Gamma_2$ is \emph{non-degenerate} if
$[\Gamma_1:\Sigma]\ge 2$ and $[\Gamma_2:\Sigma]\ge 3$, and this is sufficient to witness the non-amenability of $\Gamma$.

The group von Neumann algebra of an HNN extension HNN$(\Gamma,\Sigma,\theta)$ is precisely the HNN extension of von Neumann algebras HNN$(\rl\Gamma, \rl\Sigma,\Theta)$, associated to the triple $(\rl\Gamma, \rl\Sigma,\Theta)$, where $\Theta$ is the trace-preserving embedding $\rl\Sigma\to\rl\Gamma$ induced by $\theta$. For more details about HNN extensions of von Neumann algebras see \cite{Ue05} and \cite[Section 3]{FV10}.

\begin{definition}\label{def.Wstar-superrigid}
A countable group $\mathcal{G}$ is said to be \emph{W$^\ast$-superrigid} if for any countable group $\Lambda$ such that $\pi : \rl \Lambda \to \rl \mathcal{G}$ is a $\ast$-isomorphism, there exist a group isomorphism $\delta : \Lambda \to \mathcal{G}$, a character $\omega : \Lambda \to \T$ and a unitary $w \in \uu(\rl \mathcal{G})$ such that $$\pi(v_s) = \omega(s) \, w \, u_{\delta(s)} \, w^\ast \quad\text{for all}\;\; s \in \Lambda \; ,$$
where $(v_s)_{s \in \Lambda}$ and $(u_g)_{g \in \mathcal{G}}$ denote the canonical generating unitaries of $\rl \Lambda$, respectively
$\rl \mathcal{G}$.
\end{definition}

We are now ready to state the main theorem of the paper.

\begin{theorem}\label{thm.main-intro}
Assume that $\Gamma$ is one of the following countable groups:

\begin{enumerate}
\item a non-degenerate amalgamated free product $\ds\Gamma_1\ast_\Sigma\Gamma_2$, with $\Sigma$ malnormal in $\Gamma_1$;

\item a non-degenerate HNN extension $\rm{HNN}(\Gamma_0,\Sigma,\theta)$, with $\{\Sigma, \te(\Sigma)\}$ malnormal in $\Gamma_0$;

\item an i.c.c. weakly amenable group with positive first $\ell^2$-Betti number that admits a bound on the order of its finite subgroups.
\end{enumerate}

Consider the action of $\Gamma \times \Gamma$ on $\Gamma$ by left-right multiplication. Then the left-right wreath product group $\mathcal{G}=(\zz/n\zz)^{(\Gamma)} \rtimes (\Gamma \times \Gamma)$, with $n\in\{2,3\}$, is W$^\ast$-superrigid in the sense of Definition \ref{def.Wstar-superrigid}.
\end{theorem}

By \cite{BV97}, \cite{PT07}, a countable group $\Gamma$ has positive first $\ell^2$-Betti number if and only if it is non-amenable and it admits an unbounded 1-cocycle into the left regular representation. Actually, throughout this paper we will only use this characterization of having positive first $\ell^2$-Betti number, without defining explicitly $\ell^2$-Betti numbers for countable groups. In \cite[Section 3]{PT07}, there are given many examples of countable groups $\Gamma$ with positive first $\ell^2$-Betti number, such as certain amalgamated free products, certain HNN extensions, hyperbolic triangle groups, limit groups of Sela, etc. Moreover, \cite[Theorem 3.2]{PT07} provides a very useful formula for estimating from below the first $\ell^2$-Betti number of a group defined by (a finite number of) generators and relations.

It is known that all Coxeter groups are weakly amenable \cite{Ja98}, \cite{Val93}. Using \cite[Theorem 3.2]{PT07} one can construct Coxeter groups with positive first $\ell^2$-Betti number and which are not hyperbolic (for details, see \cite{KN11}). Remark that such groups provide examples of groups that satisfy the third set of assumptions in Theorem \ref{thm.main-intro} and that are not covered by \cite[Theorem B]{BV12}.

\subsection*{Structure of the proof}

Let $\Gamma$ be a countable group satisfying one set of assumptions of Theorem \ref{thm.main-intro}. Denote \linebreak $H:=\zz/n\zz$, with $n\in\{2,3\}$, and consider the wreath product $\mathcal{G}:=H^{(\Gamma)} \rtimes (\Gamma \times \Gamma)$, where $\Gamma \times \Gamma$ acts on $\Gamma$ by left-right multiplication.  We want to prove that $\mathcal{G}$ is W$^\ast$-superrigid in the sense of Definition \ref{def.Wstar-superrigid}. Put $M:=\rl\mathcal{G}$ and assume that $\Lambda$ is an arbitrary countable group such that $M\cong \rl\Lambda$.

The proof follows exactly the same strategy as in \cite{IPV10} and \cite{BV12} and uses many results of these two papers. To describe more precisely this strategy, consider the comultiplication $\Delta:\rl\Lambda\to \rl\Lambda\tens \rl\Lambda$, defined by $\Delta(v_s)=v_s\otimes v_s$, for all $s\in\Lambda$, associated to the group von Neumann algebra decomposition $M\cong \rl\Lambda$. We write $A:=\rl H^{(\Gamma)}$ and $G:=\Gamma\times\Gamma$, so that $M=A\rtimes G$.

Under these assumptions, we prove that the following three statements are true:
\begin{equation}\label{main relations}\left.\begin{array}{ll} & \Delta(A)\prec^f A\tens A,\\\\
& \Delta(A)^\pr\cap M\tens M\prec^f A\tens A,\\\\
& \Omega\Delta(\rl G)\Omega^\ast\sset \rl G\tens \rl
G,\end{array}\right.\end{equation}

for some unitary $\Omega\in\uu(M\tens M)$. Here, the notation $"\prec^f"$ refers to Popa's intertwining-by-bimodules that we introduce in Section \ref{sec.prelim}.

Having these three facts established, we can literally repeat the proof of \cite[Theorem 8.1]{BV12}, followed by the proof of \cite[Theorem B]{BV12}, in the particular case $H_0=H$. This exactly yields the conclusion of Theorem \ref{thm.main-intro}.

All these statements, as well as the final argument, are showed to be true in Section 6. The rest of the paper is organized as follows. In Section 2 we introduce several preliminary notions and prove a number of technical lemmas that we need for the proof of the main theorem. In Section 3 and Section 4 we introduce the \emph{malleable deformations}, in the sense of Popa, that we can define on our wreath product group von Neumann algebra: the \emph{tensor length deformation} coming from the wreath product structure and the \emph{Gaussian deformation} coming from the 1-cocycle into the left regular representation. In Section 5 we establish results that allow us to have good control on the normalizer of relatively amenable subalgebras.

\section{Preliminaries}\label{sec.prelim}

\subsection{Popa's intertwining-by-bimodules}

Let $(M,\tau)$ be a tracial von Neumann algebra. Suppose that $p$ and $q$ are non-zero projections in $M$ and that $P \sset pMp$ and $Q\sset qMq$ are von Neumann subalgebras. We recall briefly Popa's
\emph{intertwining-by-bimodules} definition/theorem.

\begin{definition}\label{def:intertwining}
We write $P\prec_M Q$ if there exists a non-zero $P$-$Q$-bimodule
$\hh\sset pL^2(M)q$ which has finite right $Q$-dimension. We write
$P\prec_M^f Q$ if $Pp^\pr\prec Q$, for all non-zero projections
$p^\pr\in P^\pr\cap pMp$. If no confusion is possible, we simply write
$P\prec Q$ and $P\prec^f Q$.
\end{definition}

\begin{theorem}[{\cite[Theorem 2.1 and Corollary 2.3]{Po03}}]\label{thm:intertwining}
Suppose that $P$ is generated by a group of unitaries
$\mathcal{G}\sset \uu(P)$. The following statements are equivalent:
\begin{itemize}
\item $P\prec_M Q$;
\item There exist a non-zero projection $q_0\in M_n(\cz)\otimes Q$, a non-zero partial isometry \linebreak $v\in M_{1,n}(\cz)\otimes pMq$
 and a normal $\ast$-homomorphism $\theta:P\to q_0(M_n(\cz)\otimes Q)q_0$ such that $$av=v\theta(a),\rm{ for all }a\in P;$$
\item There is no sequence of unitaries $(u_n)_{n\in\nz}$ in $\mathcal{G}$ such that
$$\norm{E_Q(x^* u_n y)}_2\to 0, \rm{ for all }x,y\in pMq.$$

\end{itemize}
\end{theorem}

The next lemma is essentially a variant of \cite[Theorem A.1]{Po01}, but we give a complete proof.
\begin{lemma}\label{lem.intert-masa}
Let $M$ be a type II$_1$ factor and $A\subset M$ be a Cartan subalgebra. Let $B\subset M$ be an abelian subalgebra and $\mathcal{G}<\mathcal{N}_M(B)$ be a subgroup such that
\begin{itemize}
\item $B^\prime\cap M\prec A$,
\item the adjoint action of $\mathcal{G}$ on $Z(B^\prime\cap M)$ is ergodic.
\end{itemize}

Then there exist a projection $p\in A$ and an element $v\in \M_{1,n}(\mathbb{C})\otimes Mp$ such that $vv^*=1$, $v^*v=1\otimes p$ and $v^*(B^\pr\cap M)v=\M_n(\cz)\otimes Ap$.
\end{lemma}
\begin{proof}
Since $B^\pr\cap M\prec A$, the von Neumann algebra $B^\pr\cap M$ has a type I direct summand. Since the adjoint action of $\g$ on $Z(B^\pr\cap M)$ is ergodic, we find an integer $n\ge 1$ such that $B^\pr\cap M=\M_n(\cz)\otimes Z(B^\pr\cap M)$. So, we may take a system of matrix units $(e_{ij})_{1\le i,j\le n}$ in $B^\pr\cap M$ with $e:=e_{11}$ satisfying $e(B^\pr\cap M)e=Z(B^\pr\cap M)e$. By construction, $Z(B^\pr\cap M)e$ is a maximal abelian subalgebra of $eMe$, whose normalizer acts ergodically on $Z(B^\pr\cap M)e$.

Since $B^\pr\cap M=\M_n(\cz)\otimes Z(B^\pr\cap M)$ and $B^\pr\cap M\prec A$, it follows that $Z(B^\pr\cap M)e\prec A$ and hence, by \cite[Theorem A.1]{Po01}, there exist a projection $p\in A$ and $v_0\in \M_{n,1}(\cz)\otimes Mp$ such that $v_0v_0^*=e$, $v_0^*v_0= p$ and $v_0^*(B^\pr\cap M)v_0=Ap$. Define $v\in \M_{1,n}(\cz)\otimes Mp$ by $v=\ds\sum_{k=1}^{n}{e_{1k}\otimes e_{1k}v_0}$. Then one checks easily that $vv^*=1$, $v^*v=1\otimes p$ and $v^*(B^\pr\cap M)v=\M_n(\cz)\otimes Ap$.

\end{proof}

\subsection{Jones' basic construction}

Let $(M,\tau)$ be a tracial von Neumann algebra and $Q\sset M$ be a von Neumann subalgebra. The \emph{Jones' basic construction} for the inclusion $Q\sset M$ is defined as the von Neumann algebra $\<M,e_Q\>$ generated by $M$ and the orthogonal projection $e_Q:\rl^2( M)\to\rl^2(Q)$.

We list now the main properties of the basic construction. Denote by $Me_QM$ the linear span of the set $\{xe_Qy\mid x,y\in M\}$.

If $Q$ is a von Neumann subalgebra of a tracial von Neumann algebra $(M,\tau)$, then the basic construction $\<M,e_Q\>$ is a semifinite von Neumann algebra, with a faithful normal semifinite trace $\Tr$, satisfying the following properties:
\begin{itemize}
\item  $\<M,e_Q\>$ equals the commutant of the right action of $Q$ on $\rl^2(M)$ and the $\ast$-subalgebra $Me_QM$ is weakly dense in $\<M,e_Q\>$;
\item $\Tr(xe_Qy)=\tau(xy)$, for all $x,y\in M$;
\item $e_Q x e_Q=E_Q(x)e_Q=e_Q E_Q(x)$, for all $x\in M$;
\item the central support of $e_Q$ in $\<M,e_Q\>$ is 1;
\item $Me_QM$ is dense in $\rl^2(\<M,e_Q\>)$ in $\norm{\cdot}_{2,\Tr}$-norm.
\end{itemize}

Part of these properties characterize the basic construction, as in the following well-known result (see e.g. \cite[Theorem 3.3.15]{SS08}).

\begin{theorem}\label{lem.basic-cons}
Let $N$ be a semifinite von Neumann algebra with a faithful normal semifinite trace $\Tr$ and von Neumann subalgebras $Q\sset M\sset N$. Assume that $e\in N$ is a projection such that
\begin{enumerate}
\item $N$ is the weak closure of the $\ast$-subalgebra $MeM$;
\item $\Tr(e)=1$ and $\tau(x):=\Tr(xe)$ defines a faithful normal trace $\tau$ on $M$;
\item $eNe=Qe=eQ$;
\end{enumerate}
Then there is a trace-preserving $\ast$-isomorphism $\theta:\<M,e_Q\>\to N$ with $\theta(x)=x$, for all $x\in M$, and $\theta(e_Q)=e$.
\end{theorem}

\begin{lemma}\label{lem.basic-cons-gms}
Let $\Gamma\act (A,\tau)$ be a trace-preserving action of a countable group $\Gamma$ on a tracial von Neumann algebra $(A,\tau)$. Let $\Sigma<\Gamma$ be a subgroup and denote $M:=A\rtimes\Gamma$ and $Q:=A\rtimes\Sigma$. Then the basic construction $\<M, e_Q\>$ is isomorphic to $N=(A\tens \ell^\infty(\Gamma/\Sigma))\rtimes \Gamma$, where $\Gamma$ acts diagonally on $A\tens \ell^\infty(\Gamma/\Sigma)$.
\end{lemma}
\begin{proof}

Notice that $N=(A\tens \ell^\infty(\Gamma/\Sigma))\rtimes \Gamma$ is a semifinite von Neumann algebra, with a  faithful normal semifinite trace $\Tr_N$ induced by the trace on $A\tens \ell^\infty(\Gamma/\Sigma)$, and that $Q$ and $M$ can be naturally seen as subalgebras of $N$.

Define the projection $e=1\otimes\de_{e\Sigma}\in N$.   One can easily check that $N$ is the weak closure of the *-subalgebra $MeM$ and that $eNe=Qe=eQ$. Also, we have that $\Tr_N(e)=1$ and that the restriction of $\Tr_N(\cdot e)$ to $M$ equals the trace on $M$. Therefore, by Lemma \ref{lem.basic-cons}, $N$  is isomorphic to the basic construction $\< M,e_Q \>$.
\end{proof}

\subsection{Relative amenability}

\begin{definition}[{\cite[Section 2.2]{OP07}}]
Let $(M,\tau)$ be a tracial von Neumann algebra, $p\in M$ be a
non-zero projection and let $P\subset pMp$ and $Q\subset M$ be von
Neumann subalgebras. We say that $P$ is \emph{amenable relative to} $Q$
inside $M$ if there exists a $P$-central positive functional on the
basic construction $p\langle M,e_Q\rangle p$, whose restriction to
$pMp$ equals the trace $\tau$.

Following \cite{IPV10}, we say that $P$ is \emph{strongly non-amenable
relative to} $Q$ inside $M$ if, for all non-zero projections $q\in
P^\pr\cap pMp$, we have that $Pq$ is non-amenable relative to $Q$
inside $M$.
\end{definition}

\begin{definition}
Let $(M,\tau)$ and $(N,\tau)$ be tracial von Neumann algebras. Let
$P \subset M$ be a von Neumann subalgebra. An $M$-$N$-bimodule
$\bim{M}{\hh}{N}$ is said to be \emph{left $P$-amenable} if $B(\hh) \cap
(N\op)'$ admits a $P$-central state whose restriction to $M$ equals
the trace $\tau$.
\end{definition}

For more details about relative amenability and about left
amenability for bimodules, see \cite{Si10} and \cite{PV11}. The link
between these two notions is spelt out in the following remark.

\begin{remark}\label{rem:rel. amen. - left amen. bimodule}
If $(M,\tau)$ is a tracial von Neumann algebra and $P\subset pMp$
and $Q\subset M$ are von Neumann subalgebras, then by definition,
$P$ is amenable relative to $Q$ inside $M$ if and only if the
$pMp$-$Q$-bimodule $p\rl^2(M)$ is left $P$-amenable.
\end{remark}

The following criterion for relative amenability is due to
\cite{OP07} (see also \cite[Section 2.5]{PV11}). Here we copy the
formulation of \cite[Lemma 2.10]{BV12}.

\begin{lemma}[{\cite[Corollary 2.3]{OP07}}]\label{lema:criterion for left am bimodule}
Let $(M,\tau)$ be a tracial von Neumann algebra and \linebreak $P \subset pMp$
be a von Neumann subalgebra. Let $\hh$ be a $pMp$-$M$-bimodule.
Assume that \linebreak $(\xi_i)_{i\in I} \in \hh$ is a net of vectors
satisfying the following three conditions:
\begin{itemize}
\item $\ds\limsup_{i\in I}\norm{\xi_i} >0$;
\item $\ds\limsup_{i\in I}\norm{x \xi_i} \leq \norm{x}_2$, for all $x \in
pMp$;
\item $\ds\lim_{i\in I}\norm{a \xi_i - \xi_i a} = 0$, for all $a\in P$.
\end{itemize}
Then there exists a non-zero projection $q \in P' \cap pMp$ such that
the $qMq$-$M$-bimodule $q \hh$ is left $Pq$-amenable.
\end{lemma}

\begin{lemma}\label{lem.centrality-of-states}
Let $(M,\tau)$ be a tracial von Neumann algebra and assume that $M\sset\Mtil$, for some von Neumann algebra $\Mtil$. Let $S\sset M$ be a subset and let $\Omega$ be a positive functional on $\Mtil$ such that the restriction of $\Omega$ to $M$ is bounded by  $c\tau$, for some constant $c>0$. If $\Omega$ is $S$-central, then $\Omega$ is $S^{\bi}$-central.
\end{lemma}
\begin{proof}
 For all elements $y\in\Mtil $ and $x\in M$, by the Cauchy-Schwarz inequality, we have that $\abs{\Omega(yx)}^2\le\Omega(y^*y)\Omega(x^*x)\le c\Omega(y^*y)\tau(x^*x) \le c\norm{y}^2\cdot\norm{x}^2_2$ and similarly $\abs{\Omega(xy)}^2\le c\norm{y}^2\cdot\norm{x}^2_2$.

 Thus, the set $M_0:=\{x\in M\mid \Omega(xy)=\Omega(yx)\rm{ for all }y\in\Mtil\}$ is an $\rl^2$-closed $\ast$-subalgebra of $M$. Since $S$ is contained in $M_0$ and $M_0$ is $\rl^2$-closed, it follows that $S^{\bi}$ is also contained in $M_0$, and this exactly means that $\Omega$ is $S^{\bi}$-central.
\end{proof}

\begin{lemma}\label{lem.rel-ameab}
Let $\sigma:\Gamma\act (X,\mu)$ be a free p.m.p. action of a countable group $\Gamma$ on a standard probability space $(X,\mu)$. Denote $A:=\rl^\infty(X,\mu)$ and let $p\in A$ be a non-zero projection. Let $\Sigma<\Gamma$ be a subgroup, $n \ge 1$ be an integer and denote $M:=\M_n(\cz)\otimes p(A\rtimes \Gamma)p$ and $Q:=\M_n(\cz)\otimes p(A\rtimes \Sigma)p$. Assume that $\g<\uu(M)$ is a subgroup and $q\in\g^\pr\cap M$ is a non-zero projection such that
\begin{itemize}
\item $\g$ normalizes $\M_n(\cz)\otimes Ap$,
\item $(\g q)^{\bi}$ is amenable relative to $Q$.
\end{itemize}
Denote by $M_0$ the von Neumann algebra generated by $\g$ and $1\tens Ap$. Then there exists a non-zero projection $q_0\in M_0 ^ \pr \cap M$ such that $M_0 q_0$ is amenable relative to $Q$.
\end{lemma}

\begin{proof}
Since $(\g q)^{\bi}$ is amenable relative to $Q$, there exists a state $\Omega_1$ on $q\<M,e_Q\>q$ such that $\Omega_1$ is $\g q$-central and it restricts to the trace on $qMq$.

Denote $N:=\M_n(\cz)\otimes (p\otimes 1)((A\tens\ell^\infty(\Gamma/\Sigma))\rtimes\Gamma)(p\otimes 1)$, where $\sigma:\Gamma \act A\tens\ell^\infty(\Gamma/\Sigma)$ is the diagonal action. By Lemma \ref{lem.basic-cons-gms} it follows that $N$ is isomorphic with the basic construction $\<M,e_Q\>$, thus $\Omega_1$ is a $\g q$-central state on $qNq$ whose restriction to $qMq$ equals the trace.

Define a state $\Omega$ on $N$ by the formula $\Omega(T)=\Omega_1(qTq)$, for all $T\in N$. Since $q$ commutes with $\g$, it follows immediately that $\Omega$ is $\g$-central. Since $\Omega_1$ restricts to the trace on $qMq$, we get that $\Omega\mid_M$ is bounded by a multiple of the trace.

Denote $D:=\M_n(\cz)\tens Ap\tens \ell^\infty(\Gamma/\Sigma)$ and let $E_D:N\to D$ be the unique trace-preserving conditional expectation.

 We claim that every unitary $v\in \uu(M)$ that normalizes $\M_n(\cz)\otimes Ap$ also normalizes $D$ inside $N$. Indeed, take $v\in\nn_M(\M_n(\cz)\otimes Ap)$. For every $g\in \Gamma$, there exist a projection $p_g\in Ap$ and a unitary $v_g\in \uu(\M_n(\cz)\otimes A\si_g(p_g))$ such that $\sum_{g\in\Gamma}{p_g}=1$, $\sum_{g\in\Gamma}{\si_g(p_g)}=1$ and $v(1\otimes p_g)=v_g(1\otimes u_g)$. If $x\in\M_n(\cz)\otimes Ap_g$, then it follows immediately that $vxv^*= v_g(\id\otimes\sigma_g)(x)v_g^*$. Moreover, for every $x\in \M_n(\cz)\tens Ap_g\tens \ell^\infty(\Gamma/\Sigma)$, we get that $vxv^*= (v_g\otimes 1)(\id\otimes\sigma_g)(x)(v_g^*\otimes 1)$, and since the right hand side belongs to $D$, our claim is proven.

If $v\in\nn_M(\M_n(\cz)\otimes Ap)$, then $v\in \nn_N(D)$ and since $E_D$ is the unique trace-preserving conditional expectation form $N$ onto $D$, it follows that $E_D(vTv^*)=vE_D(T)v^*$, for all $T\in N$. In particular, $E_D(vTv^*)=vE_D(T)v^*$, for all $v\in \g$ and $T\in N$.

Define a state $\widetilde{\Omega}$ on $N$ by $\widetilde{\Omega}(T)=\Omega(E_D(T))$, for all $T\in N$. Since $\Omega$ is $\g$-central, the previous remark implies that $\widetilde{\Omega}$ is also $\g$-central. Since $E_D(M)\sset \M_n(\cz)\otimes Ap$, we have that $\widetilde{\Omega}\mid_M$ is bounded by a multiple of the trace. Notice that $\widetilde{\Omega}$ is automatically $(1\tens Ap)$-central, since $1\tens Ap$ commutes with $D$, and hence, by Lemma \ref{lem.centrality-of-states}, it follows that $\widetilde{\Omega}$ is an $M_0$-central state on $N$ whose restriction to $M$ is bounded by a multiple of the trace. In particular, $\widetilde{\Omega}$ is normal on $M$, and then, by \cite[Lemma 2.9]{BV12}, there exists a non-zero projection $q_0\in M_0 ^ \pr \cap M$ such that $M_0 q_0$ is amenable relative to $Q$.
\end{proof}

The next lemma is entirely contained in the proof of \cite[Theorem 7.1]{Io12a} and \cite[Lemma 8.2]{DI12}.
\begin{lemma}\label{lem.coamenable AFP HNN}
Let $\Gamma$ be a non-degenerate amalgamated free product $\Gamma=\Gamma_1*_\Sigma\Gamma_2$ or a non-degenerate HNN extension $\Gamma=\rm{HNN}(\Gamma_0,\Sigma,\theta)$. Then $\Sigma$ is not co-amenable in $\Gamma$.
\end{lemma}

\subsection{Properties of the comultiplication}

We recall now a few useful properties of the comultiplication that
we shall use throughout the paper. Let $M$ be a II$_1$ factor and
assume that $M\cong\rl\Lambda$, for some countable group $\Lambda$.
Define the comultiplication $\Delta:\rl\Lambda \to\rl\Lambda\tens
\rl\Lambda $, associated to $\Lambda$, by $$\Delta(v_s)=v_s\otimes
v_s, \rm{ for all }s\in\Lambda.$$

The next proposition is contained in \cite[proposition 7.2]{IPV10}
and \cite[Proposition 4.1]{BV12}.
\begin{proposition}\label{prop:comultiplication}

Assume that $\mtil$ is a tracial von Neumann algebra such that
$M\sset\mtil$.
\begin{enumerate}
\item If $Q\sset M$ is a von Neumann subalgebra such that $\Delta(M)\prec  M \tens $Q, then there exists a non-zero projection $q\in Q^\pr \cap M$ such that $Qq\sset qMq$
has finite index.
\item If $Q\sset\mtil$ is a von Neumann subalgebra such that $\Delta(M)$ is amenable relative to $M\tens Q$ inside $M\tens\mtil$, then $M$ is amenable relative to
$Q$ inside $\mtil$.
\item If $Q\sset M$ is a von Neumann subalgebra that has no amenable direct summand, then $\Delta(Q)$ is strongly non-amenable relative to $M\tens 1$.
\item If $Q\sset M$ is diffuse, then $\De(Q)\nprec M\tens 1$ and $\De(Q)\nprec 1\tens M$.

\item $\De(M)^\pr\cap M\tens M=\cz 1$.
\end{enumerate}
\end{proposition}

\section{Tensor length deformation}

Assume that $G$ is a countable discrete group acting on a countable
set $I$ and let $(A_0,\tau)$ be any tracial von Neumann algebra.
Consider the generalized Bernoulli action $G\act A_0^I$. Denote by
$M$ the corresponding Bernoulli crossed product $M=A_0^I\rtimes G$.

In \cite{Po03}, \cite{Po04}, Popa introduced his fundamental \emph{malleable
deformation} for Bernoulli crossed products and used it to prove the
first W$^\ast$-rigidity theorems for property (T) groups. In
\cite{Po06b}, Popa introduced spectral gap methods to prove
W$^\ast$-rigidity theorems for direct products of non-amenable groups.
These methods and results have been intensively generalized and used
in many subsequent works. For more details about Popa's
deformation/rigidity theory, we refer to the survey papers
\cite{Po06a}, \cite{Va10a}, \cite{Io12b}.

In this paper, we use the following variant of Popa's malleable
deformation for Bernoulli crossed products, due to Ioana
\cite{Io06}. Consider the free product $A_0 * \rl\zz$, with respect
to the natural traces. Denote $\mtil := (A_0\ast \rl\zz)^I \rtimes G$
the corresponding generalized Bernoulli crossed product.

Consider the self-adjoint element $h\in\rl\zz$, with spectrum
$[-\pi,\pi]$, such that $\exp(ih)$ equals the canonical generating
unitary of $\rl\zz$ and denote by $(u_t)_{t\in\rz}$ the
one-parameter group of unitary elements in $\rl\zz$ given by
$u_t:=\exp(ith)$, for all $t\in \rz$.

Define a one-parameter group of automorphisms $\al_t\in\aut(\mtil)$
by
$$\al_t(u_g)=u_g,\; \rm{ for all } g\in G,$$ and
$$\al_t(\pi_i(x))=\pi_i(u_txu_t^*),\rm{ for all } x\in A_0\ast\rl\zz,\; i\in I,$$

where $\pi_i:A_0\ast\rl\zz\to(A_0\ast\rl\zz)^I$ puts an element of
$A_0\ast\rl\zz$ in the $i$-th position in $(A_0\ast\rl\zz)^I$.

We call $(\al_t)_{t\in\rz}\in \aut(\mtil)$ the \emph{tensor length
deformation} of the generalized Bernoulli crossed product
$M=A_0^I\rtimes G$.
If $Q\sset \mtil$ is a von Neumann subalgebra, then we say that $\al_t\to \id$ uniformly on $Q$ if
$$\ds\sup_{b\in \uu(Q)}{\norm{\al_t(b)-b}_2}\to 0,\rm{ as }t\to 0.$$

%

We recall from \cite{BV12} the following \emph{spectral gap
rigidity} theorem.

\begin{theorem}[{\cite[Theorem 3.1]{BV12}}]\label{thm: spectral gap}

Let $G \curvearrowright I$ be an action of a countable group on a
countable set. Assume that $(A_0,\tau)$ and $(N,\tau)$ are arbitrary
tracial von Neumann algebras. Consider, as above, the generalized
Bernoulli crossed product $M = A_0^I \rtimes G$, with its tensor
length deformation $(\alpha_t)_{t\in\rz} \in \aut(\mtil)$. Let $p
\in N \tens M$ be a non-zero projection and  $Q \subset p (N \tens M)
p$ be a von Neumann subalgebra.

Assume there exists an integer $\kappa > 0$ such that for every
finite subset $\ff \subset I$, with $\abs{\ff} \geq \kappa$,
$$Q \rm{ is strongly non-amenable relative to } N \tens (A_0^I \rtimes \stab{\ff}).$$

Then: $$\id\otimes\alpha_t\to \id \rm{ uniformly on }\uu(Q^\pr\cap
p(N\tens M) p).$$
\end{theorem}

\section{Cocycles and Gaussian deformation}

Let $\Gamma$ be a countable group and let $c:\Gamma\to\kk_\rz$ be a
1-cocycle into the orthogonal representation $\pi:\Gamma\to
\mathcal{O}(\kk_\rz)$. The 1-cocycle $c$ defines a one-parameter
family $(\psi_t)_{t>0}$ of positive definite functions on $\Gamma$
by
$$\psi_t:\Gamma\ni g\map \psi_t(g):=\exp(-t\norm{c(g)}^2)\in \rz.$$

If $\Gamma\act (A, \tau)$ is a trace-preserving action of $\Gamma$
on the tracial von Neumann algebra $(A,\tau)$ and $M:=A\rtimes
\Gamma$ is the corresponding crossed product, then to the family
$(\psi_t)_{t>0}$ corresponds a one-parameter family $(\fy_t)_{t>0}$
of unital completely positive normal trace-preserving maps on $M$,
defined by
$$\fy_t:M\ni bu_g\map \fy_t(bu_g):=\psi(g)bu_g=\exp(-t\norm{c(g)}^2)bu_g\in M .$$

If $\kk_\rz$ is a real Hilbert space and $\pi:\Gamma\to
\mathcal{O}(\kk_\rz)$ is an orthogonal representation, then we
denote by $\kk$ the complexification of $\kk_\rz$ and by $\pi$ the
corresponding unitary representation on $\kk$. To any unitary
representation $\pi:\Gamma\to \mathcal{U}(\kk)$ we associate the
$M$-$M$-bimodule $K^\pi:=\kk\otimes \rl^2(M)$, where the left-right
$M$-module action on $\kk^\pi$ is given by
$$(bu_g)\cdot (\xi\otimes x)\cdot y=\pi(g)\xi\otimes (bu_g)xy,$$ for all $b\in A, g\in\Gamma, \xi\in \kk$ and $x,y\in M.$

The unitary representation $\pi:\Gamma\to \mathcal{U}(\kk)$ is said
to be \emph{mixing} if, for all $\xi,\eta\in\kk$, we have that
$$\langle \pi(g)\xi,\eta \rangle\to 0, \rm{ as }g\to\infty.$$

We define now the \emph{malleable Gaussian deformation}
(\cite[Section 3]{Si10}) on $M=A\rtimes \Gamma$, associated to the
1-cocycle $c:\Gamma\to\kk_\rz$ into the orthogonal representation
$\pi:\Gamma\to \mathcal{O}(\kk_\rz)$.

Denote by $\si:\Gamma\act (Y,\nu)$ the Gaussian action associated to the orthogonal representation $\pi$. Let $D:=\rl^\infty(Y,\nu)$ and $\tau$ be the
trace on $D$ given by integration with respect to $\nu$. Then $\si$ yields a trace-preserving action $(\si_g)_{g\in \Gamma}$ of $\Gamma$ on $(D,\tau)$. For the purpose
of our paper it is more convenient to see $(D,\tau)$ as the unique abelian tracial von Neumann algebra generated by unitaries $\om(\xi)$, with $\xi\in \kk_\rz$, subject
to the following relations:
$$\lf.\barr{l} \om(\xi+\eta)=\om(\xi)\om(\eta), \rm{ for all }\xi,\eta\in\kk_\rz;\\\\
               \om(0)=1,\; \om(\xi)^*=\om(-\xi), \rm{ for all }\xi\in\kk_\rz; \\\\
               \tau(\om(\xi))=\exp(-\norm{\xi}^2), \rm{ for all }\xi\in\kk_\rz.
\earr\rg.$$

By construction, the Gaussian action of $\Gamma$ on $(D,\tau)$ is given by
$$\si_g(\omega(\xi))=\omega(\pi(g)\xi),\; \rm{ for all } g\in \Gamma,
\xi\in \kk_\rz.$$

We denote $\widetilde{M}:=(D\tens A)\rtimes \Gamma$, where $\Gamma$
acts diagonally on $D\tens A$, and we define a one-parameter group
of automorphisms $(\beta_t)_{t\in\rz}\in\aut(\widetilde{M})$ by
$$\beta_t(x)=x, \rm{ for all }x\in D\tens A,$$ and
$$\beta_t(u_g)=(\omega(tc(g))\otimes 1)u_g,\;\rm{ for all } g\in \Gamma,t\in\rz.$$

The automorphisms $(\beta_t)_{t\in\rz}\in\aut(\widetilde{M})$ give a
malleable deformation in the sense of Popa, i.e. $\beta_t\to\id$
pointwise, as $t\to 0$, in the $\rl^2$-norm on $\mtil$.

We record for later use the following two easy lemmas.

\begin{lemma}[{\cite[Lemma 2.1]{Io11}}]\label{lema:cocycle must be bounded}
If $\beta_t\to \id$ uniformly on the unit ball of $pMp$, for some
non-zero projection $p\in M$, then the cocycle $c$ must be bounded.
\end{lemma}

\begin{lemma}\label{lem.cocycle-unif-on-delta-M}
Assume that $M=A\rtimes\Gamma$ is a type II$_1$ factor and $M\cong \rl\Lambda$ for some countable group $\La$. Define the comultiplication $\De:\rl\Lambda\to \rl\Lambda\tens \rl\Lambda$ by $\De(v_s)=v_s\otimes v_s$, for all $s\in \La$.

If $\id\otimes\beta_t\to \id$ uniformly on the unit ball of $q\De(M)q$, for some non-zero projection $q\in M\tens M$, then the cocycle $c$ must be bounded.
\end{lemma}
\begin{proof}
Let $q\in M\tens M$ be a non-zero projection and assume that $\id\otimes\beta_t\to \id$
uniformly on the unit ball of $q\De(M)q$. Since $\{v_s\}_{s\in\La}$
is a group of unitaries generating $M$, we have that
$\id\otimes\beta_t\to \id$ uniformly on $\{q\De(v_s)q=q(v_s\otimes
v_s)q\mid s\in\La\}$. A combination of \cite[Lemma 3.3]{Va10b} and  \cite[Lemma 3.4]{Va10b} yields a non-zero projection $q_1\in M$ such that $q\le
1\otimes q_1$ and $\beta_t\to \id$ uniformly on $\{q_1v_sq_1\mid
v\in\La\}$. Since the group of unitaries $\{v_s\}_{s\in\La}$ generates $M$, applying again
\cite[Proposition 3.4]{Va10b}, it follows that $\beta_t\to \id$
uniformly on the unit ball of $q_2Mq_2$, where $q_2$ is a projection in $M$ satisfying $q_1\le q_2$. Since $q_1$ is non-zero, it
follows that $q_2$ is also non-zero and then, by Lemma
\ref{lema:cocycle must be bounded}, the cocycle $c$ must be bounded.
\end{proof}

In \cite[Theorem 4.5]{Pe09} and \cite[Theorem 2.5]{CP10}, using
Peterson's techniques of unbounded derivations, it has been proven
that whenever $\pi$ is mixing and $\beta_t\to\id$ uniformly on a von Neumann subalgebra
$Q\sset M$ such that $Q\nprec A$, then $\beta_t\to\id$ uniformly on
the normalizer of $Q$. An alternative proof of this result was given
by Vaes, in \cite{Va10b}, using the Gaussian automorphisms
$(\beta_t)_{t\in \rz}$. The precise formulation of this result goes
as follows.

\begin{theorem}[{\cite[Theorem 3.10]{Va10b}}]\label{thm:chifan peterson direct}
Assume that $\pi$ is a mixing representation. Let $p\in M$ be a
projection and $Q\sset pMp$ be a von Neumann subalgebra such that
$Q\nprec A$ and such that $\beta_t\to \id$ uniformly on the unit
ball of $Qq$, for some non-zero projection $q\in Q^\pr\cap pMp$.
Denote by $P$ the normalizer of $Q$ inside $pMp$. Then
$\beta_t\to\id$ uniformly on the unit ball of $Pr$, where $r$ is the
smallest central projection in $Z(P)$ satisfying $q\le r$.
\end{theorem}

\section{Normalizers of (relatively) amenable subalgebras}

Let $\Gamma$ be a countable group and let $\Gamma\act (A,\tau)$ be a
trace-preserving action such that the crossed product
$M:=A\rtimes\Gamma$ is a type II$_1$ factor. Assume that $M\cong
\rl\Lambda$, for some countable group $\Lambda$, and define the
corresponding comultiplication $\Delta:\rl\Lambda\to \rl\Lambda\tens
\rl\Lambda$.

The following two results are direct consequences of \cite[Theorem
3.1]{PV11} and \cite[Theorems A and 4.1]{Va13} and they are very
similar to \cite[Theorem 5.1]{BV12}.

\begin{theorem}\label{norm of abelian salg:Betti}

Assume that $\Gamma$ is non-amenable, weakly amenable and it admits
an unbounded 1-cocycle $c:\Gamma\to\kk_\rz$ into a mixing orthogonal
representation $\pi:\Gamma\to\mathcal{O}(\kk_\rz)$ that is weakly
contained into the left regular representation of $\Gamma$.

Let $Q\subset M\tens M$ be a von Neumann subalgebra such that
$\Delta(M)\subset\mathcal{N}_{M\tens M}(Q)^{\pr\pr}$ and such that
$Q$ is amenable relative to $M\tens A$. Then $Q\prec^f M\tens A$.
\end{theorem}

\begin{proof}

Denote by $\kk^\pi$ the $M$-$M$-bimodule associated to $\pi$ and by
$(\fy_t)_{t\ge 0}$ the group of unital normal completely positive
maps associated to the 1-cocycle $c$.

Denote $P:=\mathcal{N}_{M\tens M}(Q)^{\pr\pr}$. By Proposition \ref{prop:comultiplication}.(5), we have that $\Delta(M)^\pr\cap M\tens M=\cz 1$
and since $\Delta(M)\sset P$, it suffices to prove that
$Q\prec M\tens A$.

Denote $\mm:=M\tens M$ and $\mathcal{A}:=M\tens A$, so that $\mm\cong
\mathcal{A}\rtimes\Gamma$. Assume, by contradiction, that $Q\nprec
\mathcal{A}=M\tens A$. Then, by \cite[Theorem 3.1]{PV11}, at least one of
the following must be true:
\begin{enumerate}
  \item The $\mm$-$\mm$-bimodule $\kk^\pi$ is left
  $P$-amenable;
  \item There exist $t, \de >0$ such that $\norm{\fy_t(a)}_2\ge \de$, for all $a\in \uu(Q)$.
\end{enumerate}

\bf{Case 1}. If $\bim{\mm}{\kk^\pi}{\mm}$ is left $P$-amenable, then
$\bim{\mm}{\kk^\pi}{\mm}$ is left $\Delta(M)$-amenable. Since $\pi$
is weakly contained in the left regular representation, it follows
that $\bim{\mm}{\kk^\pi}{\mm}\prec \bim{\mm}{(L^2(\mm)\otimes_\mathcal{A}
L^2(\mm))}{\mm}$, and therefore, by \cite[Corollary 2.5]{PV11}, we
get that $$\bim{\mm}{(L^2(\mm)\otimes_\mathcal{A} L^2(\mm))}{\mm} \rm{ is
left }\Delta(M)\rm{-amenable.}$$

By \cite[Proposition 2.4]{PV11} this further implies that
$\bim{\mm}{L^2(\mm)}{\mathcal{A}}$ is left $\Delta(M)$-amenable, i.e.
$\Delta(M)$ is amenable relative to $\mathcal{A}=M\tens A$. Finally, by
Proposition \ref{prop:comultiplication}.(2), we get that $M$ is
amenable relative to $A$, which contradicts the non-amenability of
$\Gamma$.

\bf{Case 2}. Assume that there exist $t, \de >0$ such that
$\norm{\fy_t(a)}_2\ge \de$, for all $a\in \uu(Q)$. Let
$(\be_t)_{t\in\rz}\in Aut(\widetilde{M})$ be the Gaussian
deformation on $M$, defined in Section 4.

Since $\pi$ is mixing, by \cite[Proposition 3.9]{Va10b}, there is a
non-zero projection $p\in Z(P)$ such that $$\id\otimes\be_t\to \id
\rm{ uniformly on the unit ball of }Qp. $$

Now, since moreover $Q\nprec \mathcal{A}$, it follows by Theorem
\ref{thm:chifan peterson direct} that
$$\id\otimes\be_t\to \id \rm{ uniformly on the unit ball of }Pq, $$ where $q\in Z(P)$ is the smallest projection such that $p\le q$.
In particular, $q$ is non-zero and since $\De(M)\sset P$ we get that $\id\otimes\be_t\to \id$ uniformly on the unit ball of $\De(M)q$, but this contradicts Lemma \ref{lem.cocycle-unif-on-delta-M}.
\end{proof}

\begin{theorem}\label{norm of abelian salg:AFP-HNN}
Assume that $\Gamma$ is an amalgamated free product
$\Gamma_1\ast_\Sigma\Gamma_2$ or an HNN extension
$\rm{HNN}(\Gamma_0,\Sigma,\theta)$ as in Theorem
\ref{thm.main-intro}.(1), respectively \ref{thm.main-intro}.(2).

Let $Q\subset M\tens M$ be a von Neumann subalgebra such that
$\Delta(M)\subset\mathcal{N}_{M\tens M}(Q)^{\pr\pr}$ and such that
$Q$ is amenable relative to $M\tens A$. Then $Q\prec^f M\tens A$.
\end{theorem}
\begin{proof}

Denote $P:=\nn_{M\tens M}(Q)^{\bi}$ and $\mathcal{A}:=M\tens A$. Suppose
first that $\Gamma=\Gamma_1\ast_\Sigma\Gamma_2$ is non-degenerate and
$\Sigma$ is malnormal in $\Gamma_1$ and notice that $\Sigma$ is
relatively malnormal in $\Gamma$ (indeed, $\Gamma_2$ has infinite
index in $\Gamma$ and $\Sigma\cap g\Sigma g^{-1}$ is finite, for all
$g\in \Gamma\setminus\Gamma_2$). Remark also that we can write
$M\tens M$ as an amalgamated free product
$$M\tens M=(\mathcal{A}\rtimes\Gamma_1)\ast_{\mathcal{A}\rtimes\Sigma}(\mathcal{A}\rtimes\Gamma_2).$$

By \cite[Theorem A]{Va13}, at least one of the following statements
is true:
\begin{itemize}
\item $Q\prec \mathcal{A}\rtimes\Sigma$;
\item $P\prec \mathcal{A}\rtimes\Gamma_i$, for some $i=1$ or $2$;
\item $P$ is amenable relative to $\mathcal{A}\rtimes \Sigma$.
\end{itemize}
If $Q\prec \mathcal{A}\rtimes\Sigma$, then we get that $Q\prec \mathcal{A}$. Indeed,
since $\Sigma$ is relatively malnormal in $\Gamma$ there is an
infinite index subgroup $\Lambda<\Gamma$ such that $\Sigma\cap
g\Sigma g^{-1}$ is finite, for all $g\in\Gamma\setminus\Lambda$.
Assume, by contradiction, that $Q\nprec \mathcal{A}$. Then, by \cite[Lemma
6.4]{Va10b}, it follows that $P\prec \mathcal{A}\rtimes\Lambda$, and hence
$\Delta(M)\prec \mathcal{A}\rtimes\Lambda$, which, by Proposition
\ref{prop:comultiplication}.(1), is not possible since $\Lambda$ has
infinite index in $\Gamma$. Thus we get that $Q\prec \mathcal{A}=M\tens A$.
By Proposition \ref{prop:comultiplication}.(5), we have that $\Delta(M)^\pr\cap M\tens
M=\cz 1$ and since moreover $\Delta(M)\sset P$, we get indeed
that $Q\prec^f M\tens A$.

If $P\prec \mathcal{A}\rtimes\Gamma_i$, for some $i=1$ or $2$, then
$\Delta(M)\prec\mathcal{A}\rtimes\Gamma_i$, which contradicts Proposition
\ref{prop:comultiplication}.(1),  since $\Gamma_i$ has infinite
index in $\Gamma$, for all $i=1,2$.

If $P$ is amenable relative to $\mathcal{A}\rtimes\Sigma$, for some $i=1$ or
$2$, then $\Delta(M)$ is amenable relative to $\mathcal{A}\rtimes\Sigma$. By
Proposition \ref{prop:comultiplication}.(2) it follows that $M$ is
amenable relative to $A\rtimes\Sigma$, but this further implies that
$\Sigma$ is co-amenable in $\Gamma$, which contradicts Lemma \ref{lem.coamenable AFP HNN}.

Assume now that $\Gamma=\rm{HNN}(\Gamma_0,\Sigma,\theta)=\<\Gamma_0, t\mid t\Sigma t^{-1}=\theta(\Sigma)\>$ is
non-degenerate and $\{\Sigma, \te(\Sigma)\}$ is malnormal in $\Gamma_0$. By \cite[Corollary 4, page 954]{KS70} we have that $\Sigma$ is malnormal in $\Gamma$, so in particular, $\Sigma$ is relatively malnormal in $\Gamma$. Using the construction in \cite[Section 3]{FV10}, we can write $M\tens M$ as an HNN extension
HNN$(\mathcal{A}\rtimes\Gamma_0,\mathcal{A}\rtimes\Sigma,\Theta)$, and hence, by
\cite[Theorem 4.1]{Va13}, at least one of the following statements
is true:
\begin{itemize}
\item $Q\prec \mathcal{A}\rtimes\Sigma$;
\item $P\prec \mathcal{A}\rtimes\Gamma_0$;
\item $P$ is amenable relative to $\mathcal{A}\rtimes \Sigma$.
\end{itemize}
The last two alternatives cannot hold, as in the previous case, thus
we have $Q\prec \mathcal{A}\rtimes\Sigma$, which implies that $Q\prec^f
M\tens A$, since $\Sigma<\Gamma$ is relatively malnormal.
\end{proof}

The next result is an analogue of \cite[Corollary 2.12]{Io12a}.
Since the first part of the proof goes exactly as in Ioana's proof,
we will be rather brief, pointing out the arguments that are
different.

\begin{theorem}\label{norm of salg amenable rel to inner part}
Assume that $\Gamma$ is non-amenable and it admits an unbounded
1-cocycle $c$ into the left regular representation of $\Gamma$.

Let $\Sigma<\Gamma$ be a subgroup and assume that the cocycle $c$ is
bounded on $\Sigma$. Denote $M_1:=A\rtimes \Sigma$ and let $Q\sset
pMp$ be a von Neumann subalgebra that is amenable relative to $M_1$,
for some non-zero projection $p \in M$. Denote
$P:=\nn_{pMp}(Q)^{\bi}$. Consider the Gaussian deformation
$(\beta_t)_{t\in\rz}\in\aut(\mtil)$ defined in Section 4. Then at
least one of the following statements holds:
\begin{itemize}
\item There is a non-zero projection $q\in Q^\pr\cap pMp$ such that $Qq$ is amenable relative to
$A$;
\item There is a  non-zero projection $r\in Z(P)$ such that $\beta_t\to \id$ uniformly on the unit ball of $Pr$.
\end{itemize} \end{theorem}
\begin{proof}

We may assume that the cocycle $c$ is zero on $\Sigma$. Since $Q$ is
amenable relative to $M_1$ inside $M$, there exists a net
$(\xi_i)_{i\in I}\in L^2(p\langle M,e_{M_1}\rangle p)$ such that
\begin{equation}\label{lema 5.3 eq 1}    \ds\lim_{i\in I}{\norm{a\xi_i-\xi_ia}_2=0,\quad \rm{for all}\; a\in
Q},\end{equation} and
\begin{equation}\label{lema 5.3 eq 2} \ds\lim_{i\in I}{\left< x\xi_i,\xi_i\right>}=\lim_{i\in I}{\langle \xi_ix,\xi_i\rangle}=\tau(x),\quad \rm{for all}\; x\in pMp. \end{equation}

Since $c$ is zero on $\Sigma$, then $\beta_t$ is identity on
$M_1=A\rtimes \Sigma$ and hence, we can extend $\beta_t$ to a trace-preserving automorphism $\beta_t$ of the basic construction $\langle \mtil,
e_{M_1}\rangle$, by letting $\beta_t(e_{M_1})=e_{M_1}$.

Denote by $\hh$ the L$^2$-closed linear span of the set
$Me_{M_1}\mtil:=\{xe_{M_1}y\;;\; x\in M, \; y\in \mtil\}$ and let
$e_\hh$ be the orthogonal projection of $\rl^2(\langle
\mtil,e_{M_1}\rangle)$ onto $\hh$.

Fix $t\in \rz$. Since, by construction, one can see $\rl^2(\<M,e_{M_1}\>)$ as a subspace of $\rl^2(\<\Mtil,e_{M_1}\>)$, we may define the net $(\xi_i^t)_{i\in I}
\sset \rl^2(\<\Mtil,e_{M_1}\>)$ by letting $\xi_i^t:=\beta_t(\xi_i)$, for all $i\in I$. We prove now that the following relations hold:
\begin{equation}\label{lema 5.3 eq 3} \ds\lim_{i\in I}{\norm{x\xi_i^t}_2}\le\norm{x}_2 \rm{ and }
\lim_{i\in I}{\norm{\xi_i^t x}_2}\le\norm{x}_2,\end{equation}
\begin{equation}\label{lema 5.3 eq 4} \ds\limsup_{i\in I}{ \norm{x e_\hh(\xi_i^t)}_2}\le
\norm{x}_2\end{equation} and
\begin{equation}\label{lema 5.3 eq 5} \ds\limsup_{i\in I}{ \norm{a\xi_i^t-\xi_i^ta}_2}\le 2\norm{a-\beta_t(a)}_2,\end{equation}
for every $a\in Q$ and for every $x\in \mtil$.

Indeed, since $\be_t$ is trace-preserving, $\xi_i\in p\hh$ and $(\mtil\ominus M)\hh\perp\hh$, by using the first part of \eqref{lema 5.3 eq 2}, we get that

$$\ds\lf.\barr{ll} \ds\lim_{i\in I}{\norm{x\xi_i^t}_2^2}&=\ds\lim_{i\in I}{\left< x\be_t(\xi_i),x\be_t(\xi_i) \right>}\\\\
& =\ds\lim_{i\in I}{\left<\be_t^{-1}(x^*x)\xi_i,\xi_i \right>}\\\\
& =\ds\lim_{i\in I}{\left< pE_M(\be_t^{-1}(x^*x))p\xi_i,\xi_i \right>}\\\\
&= \tau(pE_M(\be_t^{-1}(x^*x))p)\\\\
&= \tau(x^*x\be_t(p))\le \norm{x}_2^2. \earr\rg.$$

The second inequality of \eqref{lema 5.3 eq 3} follows similarly
using the second part of the equation \eqref{lema 5.3 eq 2}.

Now, since $(\mtil\ominus M)\hh\perp\hh$ and $\hh$ is a left
$M$-module, it follows that
$$\ds\lf.\barr{ll} \norm{xe_\hh(\xi_i^t)}_2^2&= \left< xe_\hh(\xi_i^t),xe_\hh(\xi_i^t) \right>\\\\
&=\left< E_M(x^*x)e_\hh(\xi_i^t),e_\hh(\xi_i^t)  \right>\\\\
&= \left< e_\hh(E_M(x^*x)^{1/2}\xi_i^t),e_\hh(E_M(x^*x)^{1/2}\xi_i^t) \right>\\\\
&= \norm{e_\hh(E_M(x^*x)^{1/2}\xi_i^t)}_2^2\\\\
& \le \norm{E_M(x^*x)^{1/2}\xi_i^t}_2^2, \earr\rg.$$ and hence,
passing to $\limsup$ and using \eqref{lema 5.3 eq 3}, we get that

$$\ds\limsup_{i\in I}{\norm{xe_\hh(\xi_i^t)}_2}\le \norm{E_M(x^*x)^{1/2}}_2=\norm{x}_2.$$

Finally, to prove \eqref{lema 5.3 eq 5}, we have that
$$\norm{a\xi_i^t-\xi_i^ta}_2\le \norm{(a-\be_t(a))\xi_i^t}_2+\norm{\xi_i^t(a-\be_t(a))}_2+\norm{a\xi_i-\xi_ia}_2.$$
Passing to $\limsup$ and using \eqref{lema 5.3 eq 3} and \eqref{lema
5.3 eq 1}, we get that
$$\limsup_{i\in I}{ \norm{a\xi_i^t-\xi_i^ta}_2\le 2\norm{a-\beta_t(a)}_2}.$$

For any $t>0$, consider the net
$\eta_i^t:=\xi_i^t-e_\hh(\xi_i^t)$ and denote
$\delta_i^t:=\norm{\eta_i^t}_2$. We have now two different cases
which will be treated separately.

\bf{Case 1}. Assume that there exists a $t>0$ such that
$\ds\limsup_{i\in I}{\delta_i^t}<\frac{5\norm{p}_2}{11}$.

Fix $a\in\uu(Q)$ and denote $P:=\nn_{pMp}(Q)^{\bi}$. Since
$(\mtil\ominus M)\hh\perp\hh$ and $\hh$ is a left $M$-module, it
follows that
$$\ds\lf.\barr{ll} \norm{E_M(\be_t(a))\xi_i^t}_2& \ge \norm{e_\hh(E_M(\be_t(a))\xi_i^t)}_2\\\\
&=\norm{e_\hh(\be_t(a)e_\hh(\xi_i^t))}_2\\\\
&\ge \norm{e_\hh(\be_t(a)\xi_i^t)}_2-\de_i^t\\\\
&\ge \norm{e_\hh(\xi_i^t
\be_t(a))}_2-\norm{a\xi_i-\xi_ia}_2-\de_i^t. \earr\rg.$$

On the other hand, since $\be_t$ is trace-preserving and $\hh$ is also a right $\mtil$-module, we
have that
$$\norm{e_\hh(\xi_i^t\be_t(a))}_2 =\norm{e_\hh(\xi_i^t)\be_t(a)}_2\ge \norm{\xi_i^t\be_t(a)}_2-\de_i^t=\norm{\xi_i a}_2-\de_i^t.$$
Thus $$\norm{E_M(\be_t(a))\xi_i^t}_2\ge\norm{\xi_i
a}_2-\norm{a\xi_i-\xi_ia}_2-2\de_i^t, $$ and hence, by \eqref{lema 5.3 eq 1}, \eqref{lema 5.3 eq 2} and \eqref{lema 5.3 eq 3},
$$\ds\lf.\barr{ll}\norm{E_M(\be_t(a))}_2& \ge \ds\lim_{i\in I}{\norm{E_M(\be_t(a))\xi_i^t}_2}\\\\
&\ge\ds\liminf_{i\in I}{ \left( \norm{\xi_i a}_2-\norm{a\xi_i-\xi_ia}_2-2\de_i^t \right)}\\\\
&= \norm{a}_2-2\ds\limsup_{i\in I}{\de_i^t}\\
&= \norm{p}_2-2\ds\limsup_{i\in I}{\de_i^t}>\ds\frac{\norm{p}_2}{11}. \earr\rg.$$

Therefore, for all $a\in \uu(Q)$, we have that
$$\norm{E_M(\beta_t(a))}_2> \ds\frac{\norm{p}_2}{11},$$ and hence, by \cite[Proposition
3.9]{Va10b}, there exists a non-zero projection $q_0\in Z(P)$ such
that \begin{equation}\label{lema 5.3 eq 6}  \beta_t\to \id \rm{
uniformly on the unit ball of } Qq_0.  \end{equation}

Furthermore, by \eqref{lema 5.3 eq 6} and Theorem \ref{thm:chifan
peterson direct}, it follows that
\begin{itemize}
\item either $Q\prec_M A$,
\item or $\beta_t\to \id \rm{ uniformly on the unit ball of } Pr$, where $r\in Z(P)$ is the smallest projection such that $q_0\le r$.
\end{itemize}
Note that, by \cite[Remark 2.2]{Io12a}, the first alternative yields
a non-zero projection $q\in Q^\pr\cap pMp$ such that $Qq \rm{ is
amenable relative to }A$, so the proof in Case 1 is done.

\bf{Case 2}. Suppose that, for all $t>0$, we have $\ds\limsup_{i\in
I}{\delta_i^t}\ge \frac{5\norm{p}_2}{11}$.

In this case we prove that there exists a net $(\eta_j)_{j\in
J}\sset \rl^2(\langle \mtil, e_{M_1}\rangle)\ominus \hh$ that
satisfies the following three conditions:
 \begin{equation}\label{lema 5.3 eq 7}\ds\limsup_{j\in J}{\norm{p\eta_j}}_2>0,\end{equation}
 \begin{equation}\label{lema 5.3 eq 8}\ds\limsup_{j\in J}{\norm{x\eta_j}}_2\le 2\norm{x}_2, \rm{ for all }x\in
 pMp,\end{equation} and
 \begin{equation}\label{lema 5.3 eq 9}\ds\lim_{j\in J}{\norm{a\eta_j-\eta_j a}_2}= 0, \rm{ for all }a\in Q.\end{equation}

Let $J$ denote the set of triples $j:=(X, Y, \eps)$ consisting of
finite subsets $X\sset Q$, $Y\sset pMp$ and $\eps>0$. Fix such a
triple $j=(X, Y, \eps)$. Since $\be_t$ converges to identity,
L$^2$-pointwise on $M$, we can find a $t>0$ such that, for all
$a\in Q$, we have
\begin{equation}\label{lema 5.3 eq 10}  \norm{a-\be_t(a)}_2<\eps/2 \rm{ and }\norm{p-\be_t(p)}_2<\norm{p}_2/10. \end{equation}

Let $a\in X$ and $x\in Y$. Since $\eta_i^t=(1-e_\hh)\xi_i^t$ and
$a\in Q$, we get by \eqref{lema 5.3 eq 4} that
$$\norm{a\eta_i^t-\eta_i^ta}_2\le \norm{a\xi_i^t-\xi_i^ta}_2,$$
and passing to $\limsup$ and using \eqref{lema 5.3 eq 5} and
\eqref{lema 5.3 eq 10}, it follows that
\begin{equation}\label{lema 5.3 eq 11}\ds\limsup_{i\in I}{\norm{a\eta_i^t-\eta_i^ta}_2}<\eps.\end{equation}
Moreover, by \eqref{lema 5.3 eq 3} and \eqref{lema 5.3 eq 4}, we
have that
\begin{equation}\label{lema 5.3 eq 12}\ds\limsup_{i\in I}{\norm{x\eta_i^t}_2}\le 2\norm{x}_2,\end{equation}
and by \eqref{lema 5.3 eq 3}, \eqref{lema 5.3 eq 2} and \eqref{lema
5.3 eq 10}, we also get that
\begin{equation}\label{lema 5.3 eq 13}
\ds\lf.\barr{ll}\ds\limsup_{i\in I}{\norm{p\eta_i^t}_2}&\ge \ds\limsup_{i\in I}{\left(\norm{p\xi_i^t}_2-\norm{e_\hh(\xi_i^t)}_2\right)}\\\\
&= \norm{p\be_t(p)}_2-\ds\liminf_{i\in I}{\norm{e_\hh(\xi_i^t)}_2}\\\\
&\ge\norm{p\be_t(p)}_2-\left(\norm{p}_2^2 -\ds\limsup_{i\in I}{\norm{\eta_i^t}_2^2}\right)^{1/2}\\\\
&>\left(\ds\frac{9}{10}-\ds\frac{4\sqrt{6}}{11}\right)\norm{p}_2>0.
\earr\rg.\end{equation}

Combining \eqref{lema 5.3 eq 11}, \eqref{lema 5.3 eq 12} and
\eqref{lema 5.3 eq 13} it follows that, for some $i\in I$, the
vectors $\eta_j:=\eta_i^t$ satisfy the required conditions
\eqref{lema 5.3 eq 7}, \eqref{lema 5.3 eq 8} and \eqref{lema 5.3 eq
9}.

Thus, by Lemma \ref{lema:criterion for left am bimodule}, there
exists a non-zero projection $q\in Q^\pr\cap pMp$ such that the
$qMq$-$M$-bimodule $$ q\rl^2(\langle \mtil, e_{M_1}\rangle)\ominus
\hh \rm{ is left }Qq\rm{-amenable}.$$

By the definition of $\hh$ we have that, as $M$-$M$-bimodules,
$$\rl^2(\langle \mtil, e_{M_1}\rangle)\ominus \hh \cong \rl^2(\mtil\ominus M)\otimes_{M_1}\rl^2(\mtil),$$

so, it follows that $q\rl^2(\mtil\ominus
M)\otimes_{M_1}\rl^2(\mtil)$ is left $Qq$-amenable.

By \cite[Proposition 2.4]{PV11}, it follows that the
$qMq$-$M_1$-bimodule $q\rl^2(\mtil\ominus M)$ is left $Qq$-amenable.
Since $\rl^2(\mtil\ominus M)$ is weakly contained in $\rl^2(
M)\otimes_A \rl^2( M)$ (see for instance \cite[Lemma 3.5]{Va10b}),
then by \cite[Corollary 2.5]{PV11} and \cite[Proposition 2.4]{PV11},
 we get that the $qMq$-$A$-bimodule $q\rl^2( M)$ is left $Qq$-amenable.
Thus, by Remark \ref{rem:rel. amen. - left amen. bimodule}, this
means that $Qq$ is amenable relative to $A$, for some non-zero
projection $q\in Q^\pr\cap pMp$, and this concludes the proof of
Case 2.
\end{proof}

\section{Proof of the main result}

This whole section will be devoted to prove that, in the setting we
shall describe below, the three conditions from \eqref{main
relations}  hold, and thus we may apply results in \cite{BV12} to
conclude the proof of Theorem \ref{thm.main-intro}.

Throughout this section, $\Gamma$ will be a countable group as in Theorem
\ref{thm.main-intro}, namely:
\begin{enumerate}
\item $\Gamma=\Gamma_1\ast_\Sigma\Gamma_2$ non-degenerate, with $\Sigma$ malnormal in
$\Gamma_1$;

\item $\Gamma=\rm{HNN}(\Gamma_0,\Sigma,\theta)$ non-degenerate, with $\{\Sigma, \te(\Sigma)\}$ malnormal in
$\Gamma_0$;

\item $\Gamma$ is i.c.c., weakly amenable, has positive first $\ell^2$-Betti number and admits a bound on the order of
its finite subgroups.
\end{enumerate}

Let $H=\zz/n\zz$, with $n=2$ or $3$, and denote $A_0:=\rl H$, $A:={A_0}^{(\Gamma)}$. Consider the generalized Bernoulli action $G:=\Gamma\times\Gamma \act
A$, and put $M:=A \rtimes (\Gamma\times\Gamma)$.
Notice that, by \cite[Theorem 6.1.(c)]{BV12}, $M$ is a type II$_1$
factor.

Let $\Lambda$ be countable group such that $M\cong \rl\Lambda$, and
define the comultiplication $\Delta:\rl\Lambda\to \rl\Lambda\tens
\rl\Lambda$ by $\De(v_s)=v_s\otimes v_s$, for all $s\in \La$.

Before starting the proof, we make the following remark concerning stabilizers of finite subsets of $\Gamma$, under the left-right multiplication action of
$\Gamma\times\Gamma$. Denote by $\delta$ the diagonal embedding of $\Gamma$ into $\Gamma\times\Gamma$.

\begin{remark}\label{rem:stabilizer of finite set}

Let $\Gamma$ be a countable group as above and let $s_1,\ldots, s_k\in \Gamma$ be $k$ distinct elements, where $k\ge 2$. The stabilizer of $\{s_1,\ldots, s_k\}$ under the left-right multiplication action of $\Gamma\times\Gamma$ equals $(1,s_1^{-1})\de(H_0)(1,s_1)$, where $H_0<\Gamma$ is defined as the centralizer of the $k$ distinct elements $s_is_1^{-1}$, for $i=1, \ldots, k$. Therefore, if $\ff\sset\Gamma$ is a finite subset with $\abs{\ff}\ge k$, then $\stab(\ff)$ can be conjugated into $\delta(H_0)$, where $H_0$ is the centralizer of $k$ distinct elements in $\Gamma$. Moreover, these $k$ distinct elements necessarily generate an infinite subgroup of $\Gamma$.

Suppose first that $\Gamma=\Gamma_1\ast_\Sigma\Gamma_2$ is non-degenerate and that $\Sigma$ is malnormal in $\Gamma_1$. Let $g\in \Gamma$ be a non-trivial element. By \cite[Theorem 2]{Le67}, we have that the centralizer $Z_\Gamma(g)$ of $g$ in $\Gamma$ is either infinite cyclic or can be conjugate in $\Gamma_1$ or $\Gamma_2$. More precisely, if $g$ cannot be conjugate into $\Gamma_1$ or $\Gamma_2$, then $g$ has infinite order and $Z_\Gamma(g)$ is cyclic. If $g$ can be conjugate into one of the $\Gamma_i$, for $i=1$ or $2$, but not in $\Sigma$, then also $Z_\Gamma(g)$ gets conjugate into $\Gamma_i$. If $g$ can be conjugate into $\Sigma$, then the malnormality of $\Sigma$ in $\Gamma_1$ forces $Z_\Gamma(g)$ to be conjugate into $\Gamma_2$. Thus, if $\ff\sset \Gamma$ is a finite subset and $\abs{\ff}\ge 2$, then the stabilizer of $\ff$ under the left-right multiplication action is either cyclic (and hence amenable) or it is conjugate to a subgroup of $\de(\Gamma_i)$, for some $i=1$ or $2$.

A similar argument can be done also for HNN extensions, using \cite[Theorem 9]{KS70} and its corollaries. If $\Gamma=\rm{HNN}(\Gamma_0,\Sigma,\theta)$ is a non-degenerate HNN extension with $\{\Sigma,\te(\Sigma)\}$  malnormal in $\Gamma_0$ and if $\ff\sset\Gamma$ is a finite subset with $\abs{\ff}\ge 2$, then $\stab(\ff)$ is either infinite cyclic (and hence amenable) or conjugated to a subgroup of $\de(\Gamma_0)$.

Finally, let $\Gamma$ be as in assumption (3) and denote by $\kappa$ the bound on the order
of its finite subgroups. Let $c$ be an unbounded 1-cocycle into the left regular representation of $\Gamma$. If $\ff\sset\Gamma$ is a finite subset with $\abs{\ff}\ge \kappa$, then $\stab(\ff)$ can be conjugated into $\delta(H_0)$, where $H_0$ is the centralizer of $\kappa$ distinct elements in $\Gamma$. Since these $\kappa$ distinct elements necessarily generate an infinite subgroup $H<\Gamma$ that commutes with $H_0$, by \cite[Lemma 2.5.(1)]{Io11} it follows that either $H_0$ is amenable or the cocycle $c$ is bounded on $H$. If the cocycle $c$ is bounded on $H$, then since the left regular representation of $\Gamma$ is mixing, by \cite[Lemma 2.5.(2)]{Io11}, we get that $c$ is bounded on $H_0$. Thus, for any finite subset $\ff\sset\Gamma$ with $\abs{\ff}\ge \kappa$ we have that either $H_0$ is amenable or the cocycle $c$ is bounded on $H_0$.

\end{remark}

\begin{lemma}\label{A embeds in A tens A}
Under these assumptions, we have that $\De(A)\prec^f A\tens A$.
\end{lemma}

\begin{proof}
Write $M\cong(A\rtimes (1\times\Gamma))\rtimes(\Gamma\times 1)$ and
$M\cong(A\rtimes(\Gamma\times 1))\rtimes(1\times\Gamma)$. Applying
Theorem \ref{norm of abelian salg:AFP-HNN}, respectively Theorem
\ref{norm of abelian salg:Betti}, in both cases, for the subalgebra
$\De(A)\sset M\tens M$, we get that $$\De(A)\prec^f M\tens (A\rtimes
(1\times\Gamma))\quad\rm{ and }\quad \De(A)\prec^f M\tens (A\rtimes
(\Gamma\times 1)),$$ and hence, by \cite[Lemma 2.7]{BV12},
$\De(A)\prec^f M\tens A$.

By symmetry, it also follows that $\De(A)\prec^f A\tens M $, and thus, by \cite[Lemma 2.2.(b)]{BV12}, we have that $\De(A)\prec^f A\tens A$.
\end{proof}

We prove now the following \emph{spectral gap rigidity} lemmas, which
rely on Theorem \ref{thm: spectral gap} and that are similar to
\cite[Lemma 8.8]{BV12}.

\begin{lemma}\label{lema:spectral gap for AFP and HNN}
Suppose that $\Gamma$ is an amalgamated free product
$\Gamma=\Gamma_1\ast_\Sigma\Gamma_2$ or an HNN extension
$\Gamma=$HNN$(\Gamma_0, \Sigma,\theta)$, as in assumption (1), respectively
(2). Let $Q\sset M\tens M$ be a von Neumann subalgebra and denote by
$P$ the von Neumann algebra generated by its normalizer in $M\tens
M$. Assume that $Q$ is strongly non-amenable relative to $M\tens A$
and that $\Delta(\rl G)\sset P$. Let $(\alpha_t)_{t\in\rz}$ be the
tensor length deformation on $M$ defined in Section 3. Then either
$$ \id\otimes\alpha_t\to \id \rm{ uniformly on } \uu(Q^\pr\cap M\tens M)$$
or there exists a non-zero projection $q\in P^\pr\cap M\tens M$ such that
$$ Pq \rm{ is amenable relative to }(M\tens A)\rtimes(\Gamma\times\Sigma)\rm{ or to }(M\tens A)\rtimes(\Sigma\times\Gamma).$$
\end{lemma}
\begin{proof}

We assume that $P$ is strongly non-amenable relative to $(M\tens
A)\rtimes(\Gamma\times\Sigma)$ and to $(M\tens
A)\rtimes(\Sigma\times\Gamma)$ and we prove that
$\id\otimes\alpha_t$ converges to $\id$ uniformly on $\uu(Q^\pr\cap
M\tens M)$.

Suppose first that $\Gamma=\Gamma_1\ast_\Sigma\Gamma_2$ is
non-degenerate and $\Sigma$ malnormal in $\Gamma_1$. By Remark
\ref{rem:stabilizer of finite set}, if $\ff\sset\Gamma$ is a finite
subset and $\abs{\ff}\ge 2$, then $\stab(\ff)$ is either amenable or
it is conjugated to a subgroup of $\delta(\Gamma_i)$, for some $i=1$
or $2$.

Then the lemma follows from Theorem \ref{thm: spectral gap} once we
have proven that $Q$ is strongly non-amenable relative to $M\tens
(A\rtimes\delta(\Gamma_i))$, for $i=1,2$.

Assume, by contradiction, that there exists a non-zero projection
$q\in Q^\pr\cap M\tens M$ such that $Qq$ is amenable relative to
$M\tens(A\rtimes\delta(\Gamma_i))$. Denote $\arn:=M\tens A$. By
assumption, $\Delta(\rl G)\sset P$ and moreover, by \cite[Lemma
2.6]{BV12}, we may assume that $q\in Z(P)$. Writing $M\tens M$ as an
amalgamated free product $M\tens
M=(\arn\rtimes(\Gamma\times\Gamma_1))\ast_{\arn\rtimes(\Gamma\times\Sigma)}(\arn\rtimes(\Gamma\times\Gamma_2))$
and applying \cite[Theorem A]{Va13}, at least one of the following
assertions is true:
\begin{itemize}
\item $Qq\prec\arn\rtimes (\Gamma\times\Sigma)$;
\item $Pq\prec\arn \rtimes (\Gamma\times\Gamma_i)$, for some $i=1$ or
$2$;
\item $Pq$ is amenable relative to $\arn\rtimes (\Gamma\times\Sigma)$.
\end{itemize}

If $Pq\prec\arn \rtimes (\Gamma\times\Gamma_i)$, for some $i=1$ or
$2$, then by Lemma \ref{A embeds in A tens A} and \cite[Lemma
2.3]{BV12} it follows that $M\prec A \rtimes
(\Gamma\times\Gamma_i)$, which is impossible since $\Gamma_i$ has
infinite index in $\Gamma$, for both $i=1$ and $2$. Notice that, by
assumption, the last alternative cannot hold.

If $Qq \prec\arn\rtimes (\Gamma\times\Sigma)$, then we have that
$Qq\prec \arn\rtimes(\Gamma\times 1)$. To prove this, assume that
$Qq\nprec \arn\rtimes(\Gamma\times 1)$. Since $\Sigma<\Gamma$ is
relatively malnormal, there exists an infinite index subgroup
$\Lambda<\Gamma$ such that $\abs{\Sigma\cap g\Sigma g^{-1}}<
\infty$, for all $g\in\Gamma\setminus\Lambda$ (e.g. $\Lambda=\Gamma_2$). Then, by \cite[Lemma
6.3]{Va10b}, it follows that $\Delta(\rl G)\prec
\arn\rtimes(\Gamma\times\Lambda)$ and hence, by Lemma \ref{A embeds
in A tens A} and \cite[Lemma 2.3]{BV12}, we get that $M\prec
A\rtimes(\Gamma\times\Lambda)$, which is impossible since $\Lambda$
has infinite index in $\Gamma$.

By symmetry, writing $M\tens M=(\arn\rtimes(\Gamma_1\times\Gamma))\ast_{\arn\rtimes(\Sigma\times\Gamma)}(\arn\rtimes(\Gamma_2\times\Gamma))$
and using the same arguments as above, it follows that also $Qq\prec \arn\rtimes(1\times\Gamma)$ and hence, by \cite[Lemma 2.2.(b)]{BV12}, $Qq\prec \arn$. Now this implies that there exists a non-zero projection $q^\pr \in Q^\pr\cap M\tens M$ such that $Q q^\pr$ is amenable relative to
$M\tens A$, which contradicts our initial assumption.

Suppose now that $\Gamma=\rm{HNN}(\Gamma_0,\Sigma,\theta)$ is non-degenerate and $\{\Sigma,\te(\Sigma)\}$ is malnormal in $\Gamma_0$. By Remark \ref{rem:stabilizer of finite set}, if $\ff\sset\Gamma$ is a finite subset and $\abs{\ff}\ge 2$, then $\stab(\ff)$ is either amenable or is conjugated to a subgroup of $\delta(\Gamma_0)$. Then the conclusion follows in the same manner as for amalgamated free products, using \cite[Theorem 4.1]{Va13} instead of \cite[Theorem A]{Va13}.
\end{proof}

\begin{lemma}\label{lema:spectral gap for weak amen}
Suppose that $\Gamma$ is weakly amenable and has positive first $\ell^2$-Betti number, as in assumption (3). Let
$Q\sset M\tens M$ be a von Neumann subalgebra and denote by $P$ the
von Neumann algebra generated by its normalizer in $M\tens M$.
Assume that $Q$ is strongly non-amenable relative to $M\tens A$ and
that $\Delta(\rl G)\sset P$. Let $(\alpha_t)_{t\in\rz}$ be the
tensor length deformation on $M$ defined in Section 3. Then
$$ \id\otimes\alpha_t\to \id \rm{ uniformly on } \uu(Q^\pr\cap M\tens M).$$
\end{lemma}
\begin{proof}
As we have remarked before, by Theorem \ref{thm: spectral gap}, it
suffices to prove the existence of an integer $\kappa>0$ such that
for any finite subset $\ff\sset\Gamma$, with $\abs{\ff}\ge \kappa$,
we have that
$$Q\rm{ is strongly non-amenable relative to }M\tens(A\rtimes\stab{\ff}).$$

To prove this claim, assume by contradiction, that for every integer
$\kappa>0$, there exists a finite subset $\ff\sset\Gamma$, with
$\abs{\ff}\ge \kappa$, and there exists a non-zero projection $q\in
Q^\pr\cap M\tens M$ such that \begin{equation}\label{eq 1:spec gap
lemma}Qq \rm{ is amenable relative to }M\tens
(A\rtimes\stab{\ff}).\end{equation}

Since $\Gamma$ has positive first $\ell^2$-Betti number, it is non-amenable and admits an unbounded
1-cocycle $c$ into the left regular representation. Fix $\kappa$ to be the bound on the order of finite subgroups of
$\Gamma$. By assumption, for this particular $\kappa$, there is a
finite set $\ff\sset\Gamma$, with  $\abs{\ff}\ge \kappa$, satisfying
\eqref{eq 1:spec gap lemma}. By Remark \ref{rem:stabilizer of finite
set}, we have that either $\stab{\ff}$ is amenable or the cocycle
$c$ is bounded on $\stab(\ff)$.

If $\stab{\ff}$ is amenable, then \eqref{eq 1:spec gap lemma}
implies that $Qq$ is amenable relative to $M\tens A$, which
contradicts our initial assumption.

Let $(\beta_t)_{t\in\rz}$ be the Gaussian deformation on $M$ defined
in Section 4. If the cocycle $c$ is bounded on $\stab{\ff}$, then,
by Theorem \ref{norm of salg amenable rel to inner part}, one of the
following statements must be true:

\begin{itemize}
\item There exists a non-zero projection $q^\pr\in Q^\pr\cap M\tens M$ such that $Qq^\pr$ is amenable relative to $M\tens
A$;
\item There exists a non-zero projection $r\in Z(Pq)$ such that $\id\otimes\beta_t\to\id$ uniformly on the unit ball of $Pr$.
\end{itemize}
The first alternative clearly contradicts the initial assumption. If
$\id\otimes\beta_t\to\id$ uniformly on the unit ball of $Pr$, then
since $\Delta(\rl G)q\sset\nn_{M\tens M}(Qq)^{\bi}$, it follows that
$\id\otimes\beta_t\to\id$ uniformly on the unit ball of $\Delta(\rl
G)q$. By Lemma \ref{A embeds in A tens A} we get that, in
particular, $\id\otimes\beta_t\to\id$ uniformly on the unit ball of
$\Delta(A)$. Thus $\id\otimes\beta_t\to\id$ uniformly on the set
$\{q\De(au_g)q\mid a\in\uu(A),g\in G\}$. Since $\{a u_g\mid
a\in\uu(A),g\in G\}$ generate $M$, by \cite[Lemma 3.4]{Va10b}, there
exists a non-zero projection $q_1\in\De(M)^\pr\cap M\tens M$ such
that $\id\otimes\beta_t\to\id$ uniformly on the unit ball of
$\Delta(M)q_1$, but this contradicts Lemma
\ref{lem.cocycle-unif-on-delta-M}.
\end{proof}

The next lemma is an immediate consequence of \cite[Lemma 2.4]{Io06}

\begin{lemma}\label{ioana's lemma}

Let $p\in M\tens A$ be a non-zero projection and $N\subset p(M\tens A)p$ be a von Neumann subalgebra. If there are $\de>0$ and $t>0$ such that
$\tau(w^\ast(\id\otimes \al_t)(w))\ge \de$, for all $w\in\uu(N)$, then there exists a finite subset $\ff\subset\Gamma$ such that $$N\prec M\tens A_0^{\ff}.$$
\end{lemma}

\begin{lemma}\label{rel com A embeds in A tens A}

We have that $\De(A)^\pr\cap M\tens M\prec^f A\tens A$.
\end{lemma}

\begin{proof}
Denote $Q:=\De(A)^\pr\cap M\tens M$ and $P:=\nn_{M\tens M}(Q)^{\bi}$. It suffices to prove that there
exists a non-zero projection $p\in Q^\pr\cap M\tens M$ such that
\begin{equation}\label{corner of Q is am rel to MtensA}Qp \rm{ is amenable relative to }M\tens A.\end{equation}

Indeed, suppose that there exists a non-zero projection $p\in Q^\pr\cap
M\tens M$ such that $Qp$ is amenable relative to $M\tens A$. Since
$\De(M)\sset P$ and since $\De(M)^\pr\cap
M\tens M=\cz\cdot 1$, it follows that $Q$ is amenable relative to
$M\tens A$. Applying Theorem \ref{norm of abelian salg:Betti}, respectively Theorem \ref{norm of abelian salg:AFP-HNN} for $Q\sset M\tens M = (M\tens (A\rtimes(1\times\Gamma)))\rtimes(\Gamma\times 1)$ and $Q\sset
M\tens M = (M\tens (A\rtimes(\Gamma\times
1)))\rtimes(1\times\Gamma)$, we get that $Q\prec^f M\tens A$, and by
symmetry, $Q\prec^f A\tens A$.

Thus, the only thing we need to prove is \eqref{corner of Q is am
rel to MtensA}. Assume not, i.e.
$$Q\rm{ is strongly non-amenable relative to }M\tens A.$$

\bf{Claim}: $ \id\otimes\alpha_t \to \id \rm{ uniformly on }\uu(\Delta(A))$.

Suppose first that $\Gamma$ is an amalgamated free product $\Gamma=\Gamma_1\ast_\Sigma\Gamma_2$ or an HNN extension
$\Gamma=$HNN$(\Gamma_0, \Sigma,\theta)$, as in assumption (1), respectively (2).
Since $\Delta(\rl G)\sset P$, Lemma
\ref{lema:spectral gap for AFP and HNN} implies that either
$$ \id\otimes\alpha_t\to \id \rm{ uniformly on } \uu(Q^\pr\cap M\tens M)$$
or there exists a non-zero projection $q\in P^\pr\cap M\tens M$ such that
$$ Pq \rm{ is amenable relative to }(M\tens A)\rtimes(\Gamma\times\Sigma)\rm{ or to }(M\tens A)\rtimes(\Sigma\times\Gamma).$$
If  $\id\otimes\alpha_t\to \id \rm{ uniformly on } \uu(Q^\pr\cap
M\tens M)$, then obviously $ \id\otimes\alpha_t \to \id \rm{
uniformly on }\uu(\Delta(A))$. If $Pq \rm{ is amenable relative to
}(M\tens A)\rtimes(\Gamma\times\Sigma)$ or to $(M\tens
A)\rtimes(\Sigma\times\Gamma)$, for some projection $q\in P^\pr\cap
M\tens M$, then since $\De(M)\sset P$, it follows that $\De(M)q$ is amenable relative to $(M\tens A)\rtimes(\Gamma\times\Sigma)$
or to $(M\tens A)\rtimes(\Sigma\times\Gamma)$. But both cases imply
that $\Sigma$ is co-amenable in $\Gamma$, which is not possible, by
Lemma \ref{lem.coamenable AFP HNN}.

Suppose now that $\Gamma$ is weakly amenable and has positive first $\ell^2$-Betti number, as in assumption (3).
Since $\De(\rl G)\sset P$, the claim follows immediately from Lemma \ref{lema:spectral gap for weak amen}.

Thus, we have that
$$  \id\otimes\alpha_t \to \id \rm{ uniformly on }\uu(\Delta(A)).$$

By Lemma \ref{A embeds in A tens A}, we have that $\De(A)\prec M\tens A$, i.e. there are non-zero projections $q\in\De(A)$, \linebreak $p\in M\tens A$, a non-zero partial isometry $v\in p(M\tens M)q$ and a normal $\ast$-homomorphism \linebreak $\te:\De(A)q\to p(M\tens A)p$ such that $bv=v\te(b), \rm{ for all } b\in \De(A)q.$

Denote $N:=\te(\De(A)q)\sset p(M\tens A)p$. Then $q^\pr:=v^*v\in N^\pr\cap p(M\tens A)p$ and we may assume that $p$ is the support projection of $E_{M\tens A}(q^\pr)$. Since $q^\pr\in N^\pr\cap p(M\tens A)p$ and since $\id\otimes\alpha_t \to \id$ uniformly on $\uu(\Delta(A))$, it follows that $$id\otimes \al_t\to id \rm{ uniformly on }(N)_1 q^\pr,$$ where $(N)_1$ denotes the unit ball of $N$.

Since $E_A\circ \al_t=E_A\circ \al_t\circ E_A$, we get that $id\otimes \al_t\to id$ uniformly on $E_{M\tens A}((N)_1 q^\pr)=(N)_1 E_{M\tens A}(q^\pr)$. Since $p$ is the support of $E_{M\tens A}(q^\pr)$, we finally get that
$$id\otimes \al_t\to id \rm{ uniformly on the unit ball of }N.$$
By Lemma \ref{ioana's lemma}, there exists a finite subset $\ff$ of $\Gamma$ such that $N\prec_{M\tens A} M\tens A_0^\ff$, i.e. there are non-zero projections $q_1\in N$ and $p_1\in M\tens A_0^\ff$, a non-zero partial isometry $v_1\in p_1(M\tens A)q_1$ and a $\ast$-homomorphism $\te_1: Nq_1\to p_1(M\tens A_0^\ff)p_1$ such that $xv_1=v_1\te_1(x)$, for all $x\in Nq_1$.

Notice that $vv_1$ is non-zero. Indeed, if $vv_1=0$, then $E_{M\tens A}(v^*v)v_1=E_{M\tens A}(v^*vv_1)=0$ and since $p$ is the support of $E_{M\tens A}(v^*v)$, we get that $v_1=pv_1=E_{M\tens A}(v^*v)v_1=0$, contradiction.

Therefore $vv_1\in p_1(M\tens M)q$ is a non-zero partial isometry and $\te_1\circ \te: \De(A)q\to p_1(M\tens A_0^\ff)p_1$ is a $\ast$-homomorphism satisfying $xvv_1=v\te(x)v_1=vv_1\te_1(\te(x))$, for all $x\in \De(A)q$, i.e.
$$\De(A)\prec M\tens A_0^\ff.$$

Since $A$ is diffuse, by Proposition
\ref{prop:comultiplication}.(4), we get that $\De(A)\nprec M\tens 1$ and hence, by \cite[Lemma 1.5]{Io06}, it follows  that
$\De(M)\prec M\tens (A\rtimes \stab\ff)$, but this contradicts
Proposition \ref{prop:comultiplication}.(1), since $\stab\ff$ has
infinite index in $\Gamma\times\Gamma$.

\end{proof}

\begin{lemma}\label{G embeds to G tens G}
There exists a unitary $\Omega\in\uu(M\tens M)$ such that
$$\Omega\De(LG)\Omega^* \sset LG\tens LG.$$
\end{lemma}
\begin{proof}
Let $\de:\Gamma\to\Gamma\times\Gamma$ be the diagonal embedding. Observe that we can write \linebreak $M\tens M=A_0^I\rtimes (G\times G )$,
where $G\times G$ acts on the disjoint union $I:=\Gamma \sqcup
\Gamma $ of two copies of $\Gamma$, with its corresponding tensor
length deformation given by $\al_t\otimes\al_t\in
Aut({\widetilde{M}\tens \widetilde{M}})$. The stabilizer of an
element $i\in I$ under the action of $G\times G$ is either of the
form $G\times g\delta(\Gamma)g^{-1}$ or $
g\delta(\Gamma)g^{-1}\times G$, for some element $g\in G$.

Since $G$ is an i.c.c. group, by \cite[Theorem 3.3]{BV12}, it suffices to
prove that :
$$ \De(\rl G) \nprec M\tens (A\rtimes\de(\Gamma)) $$
and $$\alpha_t \otimes \alpha_t \to \id \rm{ uniformly on }
\uu(\De(\rl G)).$$

The first condition is immediate. Indeed, if $\De(\rl G)\prec M\tens
(A\rtimes\de(\Gamma))$, then by Lemma \ref{A embeds in A tens A} and
\cite[Lemma 2.3]{BV12}, we get that $\De(M)\prec M\tens
(A\rtimes\de(\Gamma))$, and hence, by Proposition
\ref{prop:comultiplication}.(1), it follows that $\delta(\Gamma)$
has finite index in $\Gamma\times\Gamma$, which is a contradiction.

To prove the second condition, notice that, by symmetry, it suffices
to prove that $ \id\otimes\alpha_t\to id$  uniformly on
$\uu{(\De(\rl G))}$. Since every group element in $G$ is the product
of an element in $\Gamma\times 1$ and an element in
$1\times\Gamma$, again by symmetry, it suffices to prove that $
\id\otimes\alpha_t\to id$  uniformly on
$\uu{(\De(\rl(1\times\Gamma)))}$. Denote $Q:=\De(\rl(1\times
\Gamma))\subset M\tens M$ and $P:=\nn_{M\tens M}(Q)^{\bi}$. By
Proposition \ref{prop:comultiplication}.(3) it follows that $Q$ is
strongly non-amenable relative to $M\tens 1$ and moreover, since
$A$ is amenable, we have that $Q$ is strongly non-amenable relative
to $M\tens A$. Clearly, all the unitaries $\Delta(u_g)$, with $g\in
\Gamma\times 1$, commute with $Q$ and $\Delta(\rl G)\sset P$.

If $\Gamma$ is weakly amenable and has positive first $\ell^2$-Betti number, as in assumption (3), then the claim follows from Lemma \ref{lema:spectral gap for weak amen}.

If $\Gamma$ is an amalgamated free product $\Gamma=\Gamma_1\ast_\Sigma\Gamma_2$ or an HNN extension
$\Gamma=$HNN$(\Gamma_0, \Sigma,\theta)$, as in assumption (1), respectively (2), then Lemma \ref{lema:spectral gap for AFP and HNN} implies that either
$$ \id\otimes\alpha_t\to \id \rm{ uniformly on } \uu(Q^\pr\cap M\tens M)$$
or there exists a non-zero projection $q\in P^\pr\cap M\tens M$ such that
$$ Pq \rm{ is amenable relative to }(M\tens A)\rtimes(\Gamma\times\Sigma)\rm{ or to }(M\tens A)\rtimes(\Sigma\times\Gamma).$$

If $\id\otimes\alpha_t\to \id \rm{ uniformly on } \uu(Q^\pr\cap
M\tens M)$, then our claim is proven. To finish the proof, we show
that the second alternative gives rise to a contradiction. Note
that, since $\De(\rl G)\sset P$, it implies that $\De(\rl G)q$ is
amenable relative to $(M\tens A)\rtimes(\Gamma\times\Sigma)$ or to
$(M\tens A)\rtimes(\Sigma\times\Gamma)$.

By Lemma \ref{rel com A embeds in A tens A} we know that
$N:=\De(A)^\pr\cap M\tens M\prec A\tens A$. Then Lemma
\ref{lem.intert-masa} implies that there exist a projection $p\in
A\tens A$ and $v\in \M_{1,n}(\cz)\otimes (M\tens M)p$ such that
$vv^*=1$, $v^*v=1\otimes p$ and $v^*Nv=\M_n(\cz)\otimes (A\tens
A)p$. Note that, since $\De(A)$ is abelian, we have that
$\De(A)\sset Z(N)$ and hence $v^*\De(A)v\sset 1\tens (A\tens A)p$.
Denote $\g:=\{\De(u_g)\mid g\in G\}$ and remark that $\g$ is a group
of unitaries normalizing $N$. Since $\De(\rl G)q$ is amenable
relative to $(M\tens A)\rtimes(\Gamma\times\Sigma)$ or to $(M\tens
A)\rtimes(\Sigma\times\Gamma)$ and since $\De(M)^\pr\cap M\tens
M=\cz1$, applying Lemma \ref{lem.rel-ameab} for the group of
unitaries $v^*\g v$ normalizing $v^*Nv$, it follows that
$v^*\De(M)v$ is amenable relative to $\M_n(\cz)\tens M\tens
(A\rtimes (\Gamma\times\Sigma))$ or to $\M_n(\cz)\tens M\tens
(A\rtimes (\Sigma\times\Gamma))$. This further implies that $\De(M)$
is amenable relative to $ M\tens (A\rtimes (\Gamma\times\Sigma))$ or
to $ M\tens (A\rtimes (\Sigma\times\Gamma))$, and finally, we get
that both cases imply the co-amenability of $\Sigma$ in $\Gamma$,
which contradicts Lemma \ref{lem.coamenable AFP HNN}.
\end{proof}

\begin{lemma}\label{lem.weak mixing adjoint}
Denote $N:=\De(A)^\pr\cap M\tens M$. If $\hh\sset \rl^2(N)$ is a finite dimensional subspace that is globally invariant under the adjoint action of $(\De(g))_{g\in G}$, then $\hh\sset \cz1$.
\end{lemma}
\begin{proof}
Let $\hh\sset \rl^2(N)$ be a finite dimensional subspace, globally $(\ad\De(g))_{g\in G}$-invariant. Define $\kk\sset \rl^2(M\tens M)$ as the norm closed linear span of $\hh\De(M)$. Since $\hh$ and $\De(A)$ commute, we get that $\De(A)\kk\sset \kk$. Also, $\De(u_g)\kk\sset \kk$, for all $g\in G$, since $\hh$ is globally invariant under $(\ad\De(u_g))_{g\in G}$. Thus $\kk$ is a $\De(M)$-$\De(M)$-bimodule which, by construction, is finitely generated as a right $\De(M)$-module.

Let $s\in \Lambda$ be a non-trivial element. Since $\Lambda$ is an i.c.c. group, the centralizer of $s$ in $\Lambda$ has infinite index in $\Lambda$. Therefore, by Proposition \ref{prop:comultiplication}.(1) and \cite[Proposition 7.2.3]{IPV10}, it follows that $\kk\sset\De(\rl^2(M))$, hence $\hh\sset\De(\rl^2(M))$. Since $\De(A)$ is abelian and since $\hh$ commutes with $\De(A)$, we get that $\hh\sset \De(\rl^2(A))$. By \cite[Lemma 2.12]{BV12} the action of $G$ on $A$ is weakly mixing and since $\hh$ is finite dimensional and globally $(\ad\De(g))_{g\in G}$-invariant, we must have that $\hh\sset\cz1$.
\end{proof}

\begin{proof}[Proof of Theorem \ref{thm.main-intro}]
Let $\Gamma$ be a countable group belonging to one of the three
classes of groups in the theorem. Consider the left-right action
$\Gamma\times\Gamma\act \Gamma$ and define the generalized wreath
product $\g=H^{(\Gamma)}\rtimes (\Gamma\times\Gamma)$, where
$H:=\zz/2\zz$ or $\zz/3\zz$. Assume that $\pi:\rl\La\to\rl\g$ is a $\ast$-isomorphism, for
some countable group $\Lambda$. We want to prove that the groups $\g$ and $\Lambda$ are isomorphic and that this group isomorphism implements $\pi$, as in Definition \ref{def.Wstar-superrigid}.

Putting all the lemmas we have proven in this section together, we
get that under these assumptions, all the three relations in
\eqref{main relations} are satisfied and now we can literally repeat
the proof of \cite[Theorem 8.1]{BV12} in the particular case of
$H_0=H$. This yield an abelian group $H^\pr$ with $\abs{H}=\abs{H^\pr}$, a group isomorphism $\de:\Lambda\to \g^\pr:=(H^\pr)^{(\Gamma)}\rtimes (\Gamma\times\Gamma)$, a p.m.p. isomorphism $\theta:\widehat{H^\pr}\to\widehat{H}$, a character $\om:\g\to\T$ and a unitary $w\in\uu(\rl\g)$ such that $$\pi=\ad(w)\circ \al_\om\circ \pi_\theta\circ\pi_\de,$$
where $\pi_\de:\rl\La\to\rl\g^\pr$ is the $\ast$-isomorphism given by $\pi_\de(v_s)=u_{\de(s)}$, for all $s\in \La$, $\pi_\te:\rl\g^\pr\to\rl\g$ is the natural $\ast$-isomorphism associated with an infinite tensor product of copies of $\te$ and $\al_\om$ is the automorphism of $\rl\g$ defined by $\al_\om(u_g)=\om(g)u_g$, for all $g\in\g$.

Since $\abs{H}=\abs{H^\pr}$, we have that $H\cong H^\pr$ and we may assume that $H=H^\pr$. Thus $\g=\g^\pr$ and our initial isomorphism
$\pi:\rl\La\cong\rl\g$ is the composition of an inner automorphism $\ad(w)$, group like isomorphisms $\pi_\de$ and $\al_\om$ implemented by the group isomorphism $\de:\La\to\g$ and the character $\om$ and a $\ast$-isomorphism $\pi_\theta:\rl\g\to\rl\g$ which is not group like in general. Since $\rl H$ has dimension 2 or 3, one can check that every automorphism $\theta:\rl H\to\rl H$ is of the form $\theta=\al_\rho \circ \pi_\gamma$, where $\rho$ is a character of $H$ and $\gamma$ is a group automorphism of $H$. Then $\pi_\theta$ is group like as well, and the theorem is proven.
\end{proof}

\subsection*{Acknowledgements}

The author is grateful to his advisor Stefaan Vaes for many suggestions and comments on various drafts of this paper. The author was supported by the Research Programme G.0639.11 of the Research Foundation - Flanders (FWO).

\bibliographystyle{alpha1}

\begin{thebibliography}{0}\setlength{\itemsep}{-1mm} \setlength{\parsep}{0mm} \small

\bibitem[BV97]{BV97} {M. E. B. Bekka, A. Valette},{ Group cohomology, harmonic functions and the first L$^2$-Betti number}. \it{Potential Anal.},
\bf{6(4)}(1997), 313-326.

\bibitem[BV12]{BV12}{ M. Berbec, S. Vaes}, { W$^*$-superrigidity for group von Neumann algebras of left-right wreath products.} \it{Proc. London Math. Soc.}, \bf{108} (2014), 1116-1152.

\bibitem[Bo09a]{Bo09a} {L. Bowen}, Orbit equivalence, coinduced actions and free products. \it{Groups Geom. Dyn.} \bf{5} (2011), 1-15.

\bibitem[Bo09b]{Bo09b}{ L. Bowen}, Stable orbit equivalence of Bernoulli shifts over free groups. \it{Groups Geom. Dyn.} \bf{5} (2011), 17-38.

\bibitem[CP10]{CP10}{I. Chifan, J. Peterson,}{ Some unique group-measure space decomposition results.} \it{Duke Math. J.}, \bf{162}(2013), 1-44.

\bibitem[Co76]{Co76} {A. Connes}, Classification of injective factors. \it{ Ann. of Math. (2)} \bf{104} (1976), 73-115.

\bibitem[Co80a]{Co80a} {A. Connes}, A factor of type II$_1$ with countable fundamental group. \it{ J. Operator Theory} \bf{ 4} (1980), 151-153.

\bibitem[Co80b]{Co80b} {A. Connes}, Classification des facteurs. In \it{ Operator algebras and applications, Part 2 (Kingston, 1980)}, Proc. Sympos. Pure Math. \bf{ 38}, Amer. Math. Soc., Providence, 1982, pp.\ 43-109.

\bibitem[CH89]{CH89} {M. Cowling, U. Haagerup}, Completely bounded multipliers of the Fourier algebra of a simple Lie group of real rank one.
                     \it{Invent. Math.} \bf{96}(1989), 507-549.

\bibitem[DI12]{DI12} {I. Dabrowski, A. Ioana,} Unbounded derivations, free dilations and indecomposability results for II$_1$ factors. \it{Trans. Amer. Math. Soc.}, to appear.

\bibitem[Dy93]{Dy93}{ K. Dykema}, Free products of hyperfinite von Neumann algebras and free dimension. \it{Duke Math. J.} \bf{69}(1993), 97-119.

\bibitem[FV10]{FV10}{P. Fima, S. Vaes}, { HNN extensions and unique group measure space decomposition of II$_1$ factors.} \it{Trans. Amer. Math. Soc.} \bf{364}(2012), 2601-2617.

\bibitem[Io06]{Io06}{A. Ioana}, {Rigidity results for wreath product II$_1$ factors. }\it{ J. Funct. Anal.} \bf{252}(2007), 763-791.

\bibitem[Io11]{Io11}{A. Ioana}, {Uniqueness of the group measure space construction decomposition for Popa's $\mathcal{HT}$ factors}. \it{Geom. Funct. Anal.} \bf{22}(2012), 699-732.

\bibitem[Io12a]{Io12a} { A. Ioana}, {Cartan subalgebras of amalgamated free product II$_1$ factors. With an appendix by A. Ioana and S. Vaes.} \it{Ann. Sci. \'{E}c. Norm. Sup\'{e}r.}, to appear.

\bibitem[Io12b]{Io12b} { A. Ioana}, {Classification and rigidity for von Neumann algebras}. In \it{Proceedings of the 6th European Congress of Mathematics} (Krak\'{o}w 2012), European Mathematical Society Publishing House, 2014, pp.\ 601-625  .

\bibitem[IPV10]{IPV10}{ A. Ioana, S. Popa, S. Vaes},{ A class of superrigid group von Neumann algebras}, \it{Annals of Math.}, \bf{178}(2013), 231-186.

\bibitem[Ja98]{Ja98} {T. Januszkiewicz}, For Coxeter groups $z^{\abs{g}}$ is a coefficient of a uniformly bounded representation. \it{Fund. Math.} \bf{174(1)}(2002), 79-86.

\bibitem[KN11]{KN11}{A. Kar, G. Niblo}, Relative ends, $\rl^2$ invariants and Property (T). \it{Journal of Algebra}, \bf(333)(2011), 232-240.

\bibitem[KS70]{KS70} {A. Karrass, D. Solitar}, The free product of two groups with a malnormal amalgamated subgroup. \it{ Can. J. Math.} \bf{ 13}(1971), 933-959.

\bibitem[Le67]{Le67} {T. Lewin}, Finitely generated D-groups. \it{J. Austral. Math. Soc.} \bf{7}(1967), 375-409.

\bibitem[MvN43]{MvN43} {F.J. Murray, J. von Neumann}, Rings of operators IV, \it{ Ann. Math.} \bf{ 44}(1943), 716-808.

\bibitem[OP07]{OP07} {N. Ozawa, S. Popa}, On a class of II$_1$ factors with at most one Cartan subalgebra. \it{Ann. Math.} \bf{ 172}(2010), 713-749.

\bibitem[Pe09]{Pe09} J. Peterson,{ L$^2$-rigidity in von Neumann algebras.} \it{Invent. Math.}, \bf{175(2)}(2009), 417-433.

\bibitem[PT07]{PT07} J. Peterson, A. Thom,{ Group cocycles and the ring of affiliated operators.} \it{Invent. Math.}, \bf{185}(2011), 561-592.

\bibitem[Po01]{Po01} {S. Popa}, On a class of type II$_1$ factors with Betti numbers invariants. \emph{Ann. of Math.} \textbf{163} (2006), 809-899.

\bibitem[Po03]{Po03}{S. Popa}, Strong rigidity of II$_1$ factors arising from malleable actions of $w$-rigid groups, I. \emph{Invent. Math.} \textbf{165} (2006), 369-408.

\bibitem[Po04]{Po04} {S. Popa}, Strong rigidity of II$_1$ factors arising from malleable actions of $w$-rigid groups, II. \it{Invent. Math.} \bf{165}(2006), 409-452.

\bibitem[Po06a]{Po06a} {S. Popa}, Deformation and rigidity for group actions and von Neumann algebras. In \it{Proceedings of the International Congress of Mathematicians (Madrid, 2006)}, Vol.\ I, European Mathematical Society Publishing House, 2007, pp.\ 445-477.

\bibitem[Po06b]{Po06b} {S. Popa}, On the superrigidity of malleable actions with spectral gap. \it{ J. Amer. Math. Soc.} \bf{ 21}(2008), 981-1000.

\bibitem[PV11]{PV11} {S. Popa, S. Vaes}, Unique Cartan decomposition for II$_1$ factors arising from arbitrary actions of free groups. \it{Acta Math.} \bf{212} (2014), 141-198.

\bibitem[SS08]{SS08} {A. Sinclair, R. Smith}, Finite von Neumann algebras and masas. \it{London Mathematical Society Lecture Note Series}, \bf{351}. Cambridge University Press, Cambridge, 2008.

\bibitem[Si10]{Si10} {T. Sinclair}, Strong solidity of group factors from lattices in $\SO(n,1)$ and $\SU(n,1)$. \it{J. Funct. Anal.} \bf{260}(2011), 3209-3221.

\bibitem[Ue05]{Ue05}{Y. Ueda}, HNN extensions of von Neumann algebras. \it{J. Funct. Anal.} \bf{225}(2005), 383-426.

\bibitem[Va07]{Va07}{S. Vaes}, Explicit computations of all finite index bimodules for a family of II$_1$ factors. \it{Ann. Sci. \'{E}c. Norm. Sup\'{e}r.} \bf{41} (2008), 743-788.

\bibitem[Va10a]{Va10a} {S. Vaes}, {Rigidity for von Neumann algebras and their invariants.} In \it{Proceedings of the International Congress of Mathematicians (Hyderabad, India, 2010)}, Vol.\ III, Hindustan Book Agency, 2010, pp.\ 1624-1650.

\bibitem[Va10b]{Va10b}{S. Vaes}, One-cohomology and the uniqueness of the group measure space decomposition of a II$_1$ factor. \it{Math. Ann.}, \bf{355}(2013), 661-696.

\bibitem[Va13]{Va13}{S. Vaes}, Normalizers inside amalgamated free product von Neumann algebras. \it{Publ. Res. Inst. Math. Sci.} \bf{50}(2014), 695-721.

\bibitem[Val93]{Val93} A. Valette, Weak amenability of right-angled Coxeter groups. \it{Proc. Amer. Math. Soc.} \bf{119(4)}(1993), 1331-1334.

\end{thebibliography}

\vspace{1,5cm}

\it{Mihaita Berbec}\\
\it{KU Leuven, Department of Mathematics \\ Celestijnenlaan 200B,
B-3001 Leuven, BELGIUM}\\

mihai.berbec@wis.kuleuven.be

\end{document}